\documentclass[11pt]{amsart}
\usepackage[a4paper, margin=1.2in]{geometry}
\usepackage{amsmath, amsthm, amssymb}
\usepackage{graphicx, array, wrapfig}
\usepackage{longtable}
\usepackage{epstopdf}
\usepackage{xcolor}
\usepackage{cite}

\makeatletter 
\def\namedlabel#1#2{\begingroup
    #2%
    \def\@currentlabel{#2}%
    \label{#1}\endgroup
}
\makeatother

\newtheorem*{theorema}{Theorem A}
\newtheorem*{theorema'}{Theorem A$^\prime$}
\newtheorem*{theoremb}{Theorem B}
\newtheorem*{theoremc}{Theorem C}
\newtheorem{Thm}{Theorem}[section]
\newtheorem{Prop}[Thm]{Proposition}
\newtheorem{Lem}[Thm]{Lemma}
\newtheorem{Cor}[Thm]{Corollary}

\newtheorem*{Thm*}{Theorem}

{\theoremstyle{definition}
}
{\theoremstyle{remark}
\newtheorem{Rem}[Thm]{Remark}}

\newenvironment{Proof}{\rm \trivlist\item[\hskip \labelsep{\bf
Proof.\quad}]}{\hfill\qed\par\medskip\endtrivlist}

\newenvironment{nalign}{ 
    \begin{equation}
    \begin{aligned}
}{
    \end{aligned}
    \end{equation}
    \ignorespacesafterend
}

\newcommand{\Stab}{\operatorname{Stab}}
\newcommand{\Is}{\mathop{\mathrm{Iso}}}
\newcommand{\Gal}{\mathrm{Gal}}
\newcommand{\idq}{/\!\!{/}}
\newcommand{\G}{\mathcal G}

\newcommand{\K}{K_0}
\newcommand{\supp}{\operatorname{supp}}

\renewcommand{\o}{\{0\}}
\renewcommand{\int}{\operatorname{Int}}
\renewcommand{\S}{S^\sharp} 
\newcommand{\stein}[1]{K\G({#1})} 
\newcommand{\tight}[1]{K\G_T({#1})} 
\renewcommand{\empty}{\varepsilon} 
\newcommand{\sing}[2]{I_{#1}(#2)} 

\newcommand{\tightid}[2]{{\mathcal T}_{#1}(#2)} 
\numberwithin{equation}{section}

\title{Simplicity of inverse semigroup and \'etale groupoid algebras}

\author{Benjamin Steinberg}
\address[B.~Steinberg]{%
    Department of Mathematics\\
    City College of New York\\
    Convent Avenue at 138th Street\\
    New York, New York 10031\\
    USA}
\email{bsteinberg@ccny.cuny.edu}

\author{N\'ora Szak\'acs}
\address[N.~Szak\'acs]{%
Department of Mathematics\\ University of York \\
Heslington, York, YO10 5DD \\ UK\\}
\email{szakacsn@math.u-szeged.hu}

\thanks{The first author was partially supported by the Fulbright Commission.}

\thanks{\noindent
\setlength\intextsep{0pt}
\begin{wrapfigure}{l}{0cm}
\includegraphics[height=4em]{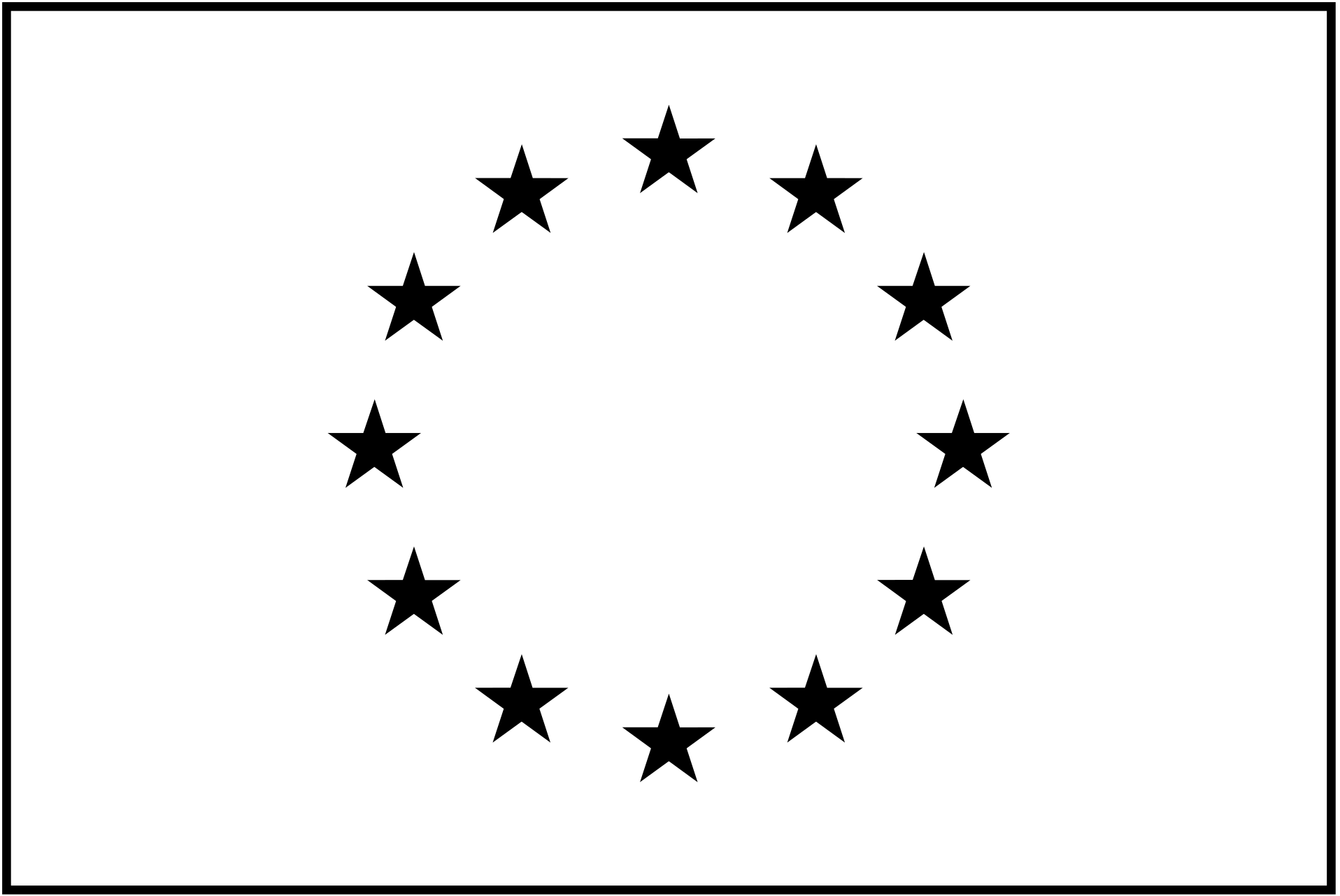}
\end{wrapfigure}
The second author was funded by the European Union’s
Horizon 2020 research and innovation programme under the
Marie Sk{\l}odowska-Curie grant agreement No 799419. \\
Part of the work was carried out during the workshop ``Higher rank graphs'' funded by the ICMS, Edinburgh.}

\date{\today}
\keywords{\'etale groupoids, inverse semigroups, simplicity, groupoid algebras, semigroup algebras}
\subjclass[2010]{20M18,20M25, 16S99,16S36, 22A22, 18B40}

\begin{document}

\begin{abstract}
In this paper, we prove that the algebra of an \'etale groupoid with totally disconnected unit space has a simple algebra over a field if and only if the groupoid is minimal and effective and the only function of the algebra that vanishes on every open subset is the null function.  Previous work on the subject required the groupoid to be also topologically principal in the non-Hausdorff case, but we do not.  Furthermore, we provide the first examples of minimal and effective but not topologically principal \'etale groupoids with totally disconnected unit spaces.  Our examples come from self-similar group actions of uncountable groups.
More generally, we show that the essential algebra of an \'etale groupoid (the quotient by the ideal of functions vanishing on every open set), is simple if and only if the groupoid is minimal and topologically free, generalizing to the algebraic setting a recent result for essential $C^*$-algebras.

The main application of our work is to provide a description of the simple contracted inverse semigroup algebras, thereby answering a question of Munn from the seventies.

Using Galois descent, we show that simplicity of \'etale groupoid and inverse semigroup algebras depends only on the characteristic of the field and can be lifted from positive characteristic to characteristic $0$.  We also provide examples of inverse semigroups and \'etale groupoids with simple algebras outside of a prescribed set of prime characteristics.
\end{abstract}

\maketitle

\section{Introduction}

This paper undertakes what we believe to be a definitive study of the simplicity of contracted inverse semigroup algebras and algebras of \'etale groupoids with totally disconnected unit spaces (also known as ample groupoids~\cite{Paterson}).   These are two interrelated subjects, each with their own rich history.

  Munn initiated the study of inverse semigroups with zero whose algebras are simple in the seventies~\cite{Munnsimplealgebra,MunnAlgebraSurvey,Munntwoexamples}.  In particular, he gave non-trivial examples which in modern terminology come from self-similar group actions on infinite alphabets and from the Leavitt path algebra of a graph with one vertex and infinitely many edges.  He also introduced in his groundbreaking work the condition on a semilattice that is equivalent to all filters being tight in the sense of Exel~\cite{Exel}, anticipating many future developments.  Munn  posed then the question of characterizing inverse semigroups with a simple algebra.  The examples considered by Munn were $0$-$E$-unitary in modern terminology, and hence have Hausdorff universal groupoids, and he imposed the condition that forced their universal groupoids to coincide with their tight groupoids and so he obtained a simplicity result that was field independent.  One gets the impression that at the time it was expected that simplicity depended only on the inverse semigroup and not on the field of coefficients.

The first author introduced algebras of ample groupoids over arbitrary coefficient rings in~\cite{mygroupoidalgebra} (nowadays termed `Steinberg algebras') in order to study inverse semigroup algebras and, in particular, to try and make progress towards Munn's question on simplicity.  The algebra of an inverse semigroup is the Steinberg algebra of its associated universal groupoid, introduced by Paterson~\cite{Paterson} to study inverse semigroup $C^*$-algebras. Conversely, all Steinberg algebras arise as particularly nice quotients of inverse semigroup algebras.  Since then, in a large part because of connections to the theory of $C^*$-algebras and Leavitt path algebras~\cite{LeavittBook}, there has been a flurry of activity concerning Steinberg algebras of ample groupoids.  See for example~\cite{Strongeffective,mygroupoidalgebra,operatorsimple1,operatorguys2,groupoidbundles,GroupoidMorita,Nekrashevychgpd,CarlsenSteinberg,ClarkPardoSteinberg,CenterLeavittGroupoid,Hazrat2017,Demeneghi,MyEffrosHahn,GonRoy2017a,mydiagonal,groupoidprimitive,groupoidprime,groupoidapproachleavitt,nonhausdorffsimple}.

In a seminal paper~\cite{operatorsimple1}, Brown~\textit{et al.} characterized Hausdorff ample groupoids with a simple complex algebra as those which are minimal and effective.  This characterization was shortly thereafter extended to arbitrary coefficient fields~\cite{operatorguys2,groupoidprimitive}.  Simplicity in the Hausdorff setting depends only on the groupoid and not the field.  Recall here that an \'etale groupoid is \emph{minimal} if all its orbits are dense in the unit space and it is \emph{effective} if the interior of the isotropy subgroupoid is the unit space.  A closely related notion is that of a topologically principal groupoid.   An \'etale groupoid is \emph{topologically principal} if the objects with trivial isotropy are dense in the unit space.  Every second countable effective \'etale groupoid is topologically principal.  Hausdorff topologically principal groupoids are effective.   Hence in the world of second countable Hausdorff groupoids, effective and topologically principal are equivalent notions.  However, topologically principal does not imply effective for non-Hausdorff groupoids.  In~\cite{operatorsimple1} an example of a Hausdorff effective groupoid was given such that each isotropy group was non-trivial, but the example was not minimal and it was unknown whether there are effective and minimal ample groupoids that are not topologically principal.

Major progress on simplicity of non-Hausdorff groupoids was obtained in~\cite{Nekrashevychgpd} and~\cite{nonhausdorffsimple}.  Nekrashevych~\cite{Nekrashevychgpd} introduced an ideal for minimal and topologically principal groupoids and proved simplicity when you factor by that ideal.  Clark~\textit{et al.}~\cite{nonhausdorffsimple}, unaware of the work of Nekrashevych, introduced an ideal, called the ideal of singular functions,
and proved that a second countable ample groupoid has a simple algebra over a field $K$ if and only if it is minimal, effective and this ideal vanishes.
Here a function is singular if it vanishes on every open set; non-zero singular functions do not exist in the Steinberg algebra of a Hausdorff groupoid.  The authors of~\cite{nonhausdorffsimple} also showed that the Nekrashevych algebra of the Grigorchuk group is simple over fields of characteristic different than $2$ but not simple over fields of characteristic $2$ (the latter was already shown in~\cite{Nekrashevychgpd}).  This was the first example showing that simplicity can depend on the field of coefficients and not just the groupoid in the non-Hausdorff setting.  They left open the question of whether there are minimal, effective and second countable groupoids whose algebras are not simple over the complex numbers\footnote{We were informed by E.~Pardo after posting the first version of this paper on ArXiv that V.~Nekrashevych has constructed such an example (unpublished).}.

It should be mentioned that simplicity of skew inverse semigroup rings is studied in~\cite{simpleskew}.  Inverse semigroup algebras and Steinberg algebras of ample groupoids are special cases of skew inverse semigroups rings, but the characterization of simple skew inverse semigroup rings is not very easy to apply, even for these special cases.  For instance, it takes a fair bit of work in~\cite{simpleskew} to recover the results of~\cite{nonhausdorffsimple} via this theory and the authors still restrict to the realm of topologically principal groupoids.

In light of the results of~\cite{nonhausdorffsimple}, a number of natural questions arise,  first and foremost of which is whether the assumption of the groupoid being topologically principal can be removed to achieve a complete characterization of simple ample groupoid algebras.  To this effect, we prove the following theorem.

\begin{theorema}
Let $\G$ be an ample groupoid and $K$ a field.  Then the Steinberg algebra $K\G$ is simple if and only if $\G$ is minimal and effective and the ideal of singular functions vanishes.
\end{theorema}

To show that this theorem has teeth, we construct the first examples of minimal and effective ample groupoids for which every isotropy group is non-trivial (and, in fact, uncountable).  Our examples come from self-similar group actions of uncountable groups.  We have  both Hausdorff and non-Hausdorff examples and the non-Hausdorff examples have simple complex algebras, but one cannot prove it using any result prior to Theorem~A.  Moreover, our examples are universal groupoids of inverse semigroups.

Recent work of Exel and Pitts~\cite{EP19} and  Kwasniewski and Meyer~\cite{KM} suggests an alternative approach to associating $C^*$-algebras to non-Hausdorff \'etale groupoids that replaces the usual reduced $C^*$-algebra by a certain quotient called the essential algebra.  The analogue in the algebraic setting would be to define the essential algebra of an ample groupoid $\G$ to be the quotient of $K\G$ by the ideal of singular functions.

A characterization of simple essential $C^\ast$-algebras was obtained in~\cite{KM}.  It involves a simultaneous weakening of the notions of effective and topologically principal groupoids.
 An \'etale groupoid is termed \emph{topologically free} in~\cite{KM} if the unit space is dense in the interior of the isotropy subgroupoid.  Effective groupoids are topologically free and the two notions coincide for Hausdorff groupoids.  Topologically principal groupoids are also topologically free and the two notions coincide for second countable groupoids~\cite{KM}.    In~\cite{KM} it is shown that minimality and topological freeness are the precise conditions that give simplicity of the essential $C^*$-algebra.  The algebraic analogue is the following result, which implies Theorem~A.

\begin{theorema'}
Let $\G$ be an ample groupoid and $K$ a field.  Then the essential algebra of $\G$ over $K$ is simple if and only if $\G$ is minimal and topologically free.
\end{theorema'}

Notice that simplicity of the essential algebra depends only on the groupoid and not the field of coefficients.  It is at the moment unclear whether the Steinberg algebra or the essential algebra is the more appropriate algebra to study for non-Hausdorff groupoids.  The fact that simplicity is field independent is certainly a strong selling point for essential algebras.  One advantage of Steinberg algebras is that it is easy to construct representations of them (using induced representations~\cite{mygroupoidalgebra,MyEffrosHahn} and sheaves~\cite{groupoidbundles}) and  Morita equivalences (using groupoid equivalences~\cite{GroupoidMorita,groupoidbundles}); they also admit nice presentations as quotients of inverse semigroup algebras.  On the other hand, it can be non-trivial to show that a representation of the Steinberg algebra factors through the essential algebra, cf.~\cite{simplicityNekr}.

While Theorem~A in some sense answers Munn's question on simple contracted inverse semigroup algebras, the fact that it relies on understanding what singular functions are, which is a topological notion, is probably not satisfactory from an inverse semigroup point-of-view.  In fact, our approach to proving Theorems~A and~A$^\prime$ is via inverse semigroup algebras.  We introduce what we call the \emph{singular ideal} of a contracted inverse semigroup algebra. It consists of all elements $a\in \K S$ such that, for each non-zero idempotent $e$ of $S$, there is a non-zero idempotent $f\leq e$ with $af=0$.  This condition is equivalent to its left-right dual.  It turns out that the singular ideal consists precisely of those elements of $\K S$ which map to a singular function of the algebra of the tight groupoid of $S$ (in the sense of Exel~\cite{Exel}); this is the only point in our arguments where topology plays a role.   Theorems~A and~A$^\prime$ are proved by working with this algebraic formulation of the singular ideal, using Exel's result that every ample groupoid is the tight groupoid of its inverse semigroup of compact local bisections.

Our second main theorem answers Munn's question on simplicity of inverse semigroup algebras in a semigroup theoretic language.  Munn observed that an inverse semigroup with zero must be congruence-free (have no proper non-zero quotients) to have a simple contracted algebra.

\begin{theoremb}
Let $S$ be an inverse semigroup with zero and $K$ a field.  Then the contracted inverse semigroup algebra $\K S$ is simple if and only if $S$ is congruence-free and the singular ideal of $\K S$ vanishes.
\end{theoremb}

Our third main theorem concerns the issue of whether simplicity of inverse semigroup and Steinberg algebras depends on the coefficient field or just the characteristic of the coefficient field, as well as what the implications are between simplicity over different fields.

\begin{theoremc}
Simplicity of the Steinberg algebra of a minimal and effective ample groupoid $\G$ depends only on the characteristic of the ground field.  Moreover, the following hold:
\begin{enumerate}
  \item if $\G$ has a simple algebra over some field of positive characteristic, then it has a simple algebra over all fields of characteristic $0$;
  \item if $\mathcal P$ is any set of primes, then there are a minimal and effective second countable ample groupoid and a countable congruence-free inverse semigroup whose algebras are simple over precisely those fields whose characteristic does not belong to $\mathcal P$;
  \item there are minimal and effective second countable ample groupoids whose algebras are not simple over any field and which are universal groupoids of an inverse semigroup; these ample groupoids also do not have a simple reduced $C^*$-algebra.
\end{enumerate}
\end{theoremc}

Theorem~C provides the first examples of congruence-free inverse semigroups whose contracted semigroup algebras are simple over some fields, but not over others.
The proof of the first part of Theorem~C relies on the technique of Galois descent. The examples in the third item  are constructed from weakly regular branch groups~\cite{Branching}, that we then make act over an infinite alphabet.

In order to prove these theorems, we prove a number of intermediate results that are of interest in their own right.  We introduce a generalization of fundamental inverse semigroups, which we call quasi-fundamental inverse semigroups, to capture topological freeness in a semigroup theoretic manner. Our most important result is a Cuntz-Krieger uniqueness theorem for the quotient of the contracted inverse semigroup algebra of a quasi-fundamental inverse semigroup (with $0$-disjunctive semilattice of idempotents) by its singular ideal.  The inverse semigroup of compact local bisections of an ample groupoid is fundamental (respectively, quasi-fundamental) if and only if the groupoid is effective (respectively, topologically free).  The tight groupoid of a Hausdorff fundamental inverse semigroup (with $0$-disjunctive semilattice of idempotents) is always effective~\cite{groupoidprimitive}, but it turns out not to be the case for non-Hausdorff inverse semigroups, as we show by an example.

The paper is organized as follows.
Section~\ref{s:prelim} recalls basic facts about inverse semigroups and ample groupoid algebras, and how these two notions are naturally intertwined. An important intermediary role is played by  Boolean inverse semigroups, which are the abstractly defined analogues of the inverse semigroups of compact local bisections of an ample groupoid; a definitive treatment can be found in~\cite{LawsonLenz}. The Steinberg algebra of an ample groupoid is a quotient of the semigroup algebra of its inverse semigroup of compact local bisections; we recall the algebraic description of the kernel, which enables us to use inverse semigroup theory to study ample groupoid algebras.
 We prove that minimality and effectiveness of the groupoid corresponds to its semigroup of compact local bisections having no non-trivial quotients in the category of Boolean inverse semigroups. This gives a transparent explanation as to why minimality and effectiveness are necessary conditions for simplicity: a proper Boolean inverse semigroup quotient gives rise to a proper quotient of Steinberg algebras.

Section~\ref{s:simpsemialg} introduces the singular ideal of an inverse semigroup algebra, proves our main uniqueness theorem for the quotient by the singular ideal and characterizes simple contracted inverse semigroup algebras, proving Theorem~B. We remark that this section can be understood without any familiarity with ample groupoids. After handling the case of inverse semigroups, we turn in Section~\ref{s:groupoidcase} to the proof of the simplicity criterion for ample groupoid algebras using Boolean inverse semigroups and the uniqueness theorem.
The following section uses descent theory to show that the property of having a non-trivial singular ideal descends from larger fields to smaller fields and from the rational numbers to finite fields, proving the first half of Theorem~C.

The final section is devoted to examples arising from inverse semigroups associated to self-similar group actions~\cite{selfsimilar,ExPadKatsura,LawsonCorrespond}. We prove the second half of Theorem~C, and also provide a counterexample to the claim of~\cite{LalondeMilan} that the tight groupoid of a fundamental inverse semigroup with $0$-disjunctive semilattice of idempotents is effective in the non-Hausdorff setting.

\section{\'Etale groupoids and Inverse Semigroups}
\label{s:prelim}

This section develops and recalls fundamental facts about inverse semigroups, ample groupoids and Boolean inverse semigroups.

\subsection{Inverse semigroups}
An \emph{inverse semigroup} is a semigroup $S$ such that for all $s\in S$, there exists a unique $s^\ast\in S$ with $ss^\ast s=s$ and $s^\ast ss^\ast=s^\ast$.   The set $E(S)$ of idempotents of $S$ forms a commutative subsemigroup.  Inverse semigroups are partially ordered via the relation $s\leq t$ if any of the following four equivalent conditions hold: $s\in tE(S)$; $s\in E(S)t$; $s=ts^\ast s$; or $s=ss^\ast t$.  The ordering is compatible with product and preserved by the involution.  The idempotents form a meet semilattice under this ordering, and the meet is given by the product.  We shall primarily be concerned with inverse semigroups containing a zero element $0$.  A good introduction to inverse semigroups is Lawson's book~\cite{Lawson}.

A \emph{filter} on a semilattice $E$ with zero is a subset $\mathcal F\subseteq E\setminus \{0\}$ closed under meets and closed upwards in the ordering.  An \emph{ultrafilter} is a maximal filter.  Zorn's lemma shows that every filter is contained in an ultrafilter and hence every non-zero element of $E$ belongs to some ultrafilter.  If $e\in E$ and $e_1,\ldots, e_n\leq e$, then one says that $e_1,\ldots, e_n$ \emph{cover} $e$ if, for all $0\neq f\leq e$, there exists $i$ with $fe_i\neq 0$.  Exel defined a filter $\mathcal F$ to be \emph{tight}~\cite{Exel} if whenever $e_1,\ldots, e_n$ cover $e$ and $e\in \mathcal F$, then $\{e_1,\ldots, e_n\}\cap \mathcal F\neq \emptyset$.  One can show that every ultrafilter is tight and that if $E$ is a Boolean algebra (not necessarily unital), then the tight filters are precisely the ultrafilters~\cite{Exel}.

A \emph{character} of a semilattice $E$ with zero is a non-zero, zero-preserving homomorphism $\chi\colon E\to \{0,1\}$ (where $\{0,1\}$ is a semigroup under usual multiplication); notice that if $K$ is a field, then any homomorphism $E(S)\to K^\times$ takes values in $\{0,1\}$ and so this notion corresponds to usual multiplicative characters in representation theory.  The characters of $E$ are precisely the characteristic functions of filters, cf.~\cite{Paterson,Exel,mygroupoidalgebra}.  \emph{Ultracharacters} are the characteristic functions of ultrafilters and  \emph{tight characters} are the characteristic functions of tight filters.  The \emph{spectrum} of $E$ is the topological space $\widehat{E}$ of characters; it is locally compact and Hausdorff in the topology of pointwise convergence. It furthermore has a basis of compact open subsets, that is, it is zero-dimensional. If $e\in E$, then put $D(e) = \{\theta\in \widehat{E}: \theta(e)=1\}$; the complement of this set is denoted $D(e)^c$.  The sets of the form $D(e)\cap D(e_1)^c\cap\cdots \cap D(e_n)^c$ with $e_1,\ldots, e_n <e$ (possibly $n=0$) are a basis of compact open sets for the topology on $\widehat{E}$.
The \emph{tight spectrum} of $E$ is the subspace $\widehat{E}_T$ of tight characters, with the induced topology;  Exel showed that the tight spectrum is the closure of the space of ultracharacters inside of $\widehat{E}$~\cite{Exel}.  Thus $\widehat{E}_T$ is also locally compact, Hausdorff and zero-dimensional.  A basis for $\widehat{E}_T$ is given by the intersection of $\widehat{E}_T$ with basic compact open sets $D(e)\cap D(e_1)^c\cap\cdots \cap D(e_n)^c$ with $e_1,\ldots, e_n <e$ not a cover of $e$ (possibly $n=0$).

An inverse semigroup $S$ is \emph{fundamental} if $Z_S(E(S)) = E(S)$, that is, the centralizer in $S$ of the semilattice of idempotents is $E(S)$; this is equivalent to saying that $E(S)$ is a maximal commutative subsemigroup of $S$.  Alternatively, $S$  is fundamental if it admits no idempotent-separating congruence except the equality relation, that is, any homomorphism which is injective on $E(S)$ is injective on $S$.  The largest idempotent separating congruence $\mu$ on $S$ is given by $s\mathrel{\mu} t$ if and only if $ses^\ast=tet^\ast$ for all $e\in E(S)$.  Then $S$ is fundamental if and only if $\mu$ is the equality relation.  The reader is referred to~\cite[Chapter~5.2]{Lawson} for details.

We shall need a weakening of the notion of fundamental that will have a natural interpretation in the groupoid context.  We say that an inverse semigroup $S$ with zero is \emph{quasi-fundamental} if each non-zero element of $Z_S(E)$ is above a non-zero idempotent, that is, if $0\neq s\in Z_S(E(S))$, then there exists $0\neq e\in E(S)$ with $e\leq s$.  It is immediate that a fundamental inverse semigroup is quasi-fundamental but the converse is false.
For example, let $G$ be any non-trivial group and $z\notin G$. Put a semigroup structure on $S=G\cup \{z,0\} $ by putting $z^2=z$ and $gz=z=zg$ for all $g\in G$ (and, of course, $0$ is a zero and $G$ is a submonoid).  Then $S$ is quasi-fundamental as $z$ is below all non-zero elements but $S$ is not fundamental as $S=Z_S(E(S))$.
The next lemma shows that an inverse semigroup is quasi-fundamental if and only if whenever two non-zero elements are $\mu$-equivalent, they have a common non-zero lower bound.

\begin{Lem}\label{l:qfundamental}
Let $S$ be an inverse semigroup with zero.  Then $S$ is quasi-fundamental if and only if whenever $0\neq s,t\in S$ with $ses^\ast=tet^\ast$ for all idempotents $e\in E(S)$, there is $0\neq u\leq s,t$.
\end{Lem}
\begin{Proof}
Assume first that $S$ is quasi-fundamental and $s \mathrel{\mu} t$.
Since $0\neq t=t(t^\ast t)t^\ast t =s(t^\ast t)s^\ast t$, we deduce that $s^\ast t\neq 0$.
Note that $s^\ast t\in Z_S(E(S))$.  Indeed, if $e\in E(S)$, then $s^\ast te= s^\ast tet^\ast t= s^\ast ses^\ast t =es^\ast t$.  Thus, by the definition of quasi-fundamental, there is $0\neq f\leq s^\ast t$.  Then $sf = ss^\ast tf=sfs^\ast t=tft^\ast t= tf$ (as $s^\ast t\in Z_S(E(S))$) and $0\neq f=s^\ast tf$ and so $tf\neq 0$.  Thus if $u=tf$, then $0\neq u\leq s,t$.

Conversely, if $s\in Z_S(E(S))$ is non-zero, then $ses^\ast = ss^\ast e=ss^\ast e(ss^\ast)^\ast$ for all $e\in E(S)$ and so by assumption, there exists $0\neq u\leq s,ss^\ast$.  But then $u\in E(S)$ and so $S$ is quasi-fundamental.
This completes the proof.
\end{Proof}

An inverse semigroup $S$ with zero is \emph{$0$-simple} if it has no proper non-zero ideals; equivalently, $S$ is $0$-simple if $SsS=S$ for all $s\neq 0$.  Recall that an \emph{ideal} of a semigroup $S$ is a non-empty subset $I$ such that $SI\cup IS\subseteq I$.  The \emph{Rees quotient} $S/I$ is the quotient of $S$ by the congruence identifying all elements of $I$ and performing no other identifications.   A semilattice $E$ is \emph{$0$-disjunctive} if $0\neq f<e$ implies there exists $0\neq g\leq e$ with $fg=0$; in other words, no element admits a cover by a single element below it.  A non-empty inverse semigroup $S$ is \emph{congruence-free} if it admits no congruences except the equality relation and the universal relation; in other words, every homomorphism $\varphi\colon S\to T$ is either injective or satisfies $|\varphi(S)|=1$.  This property cannot be described in terms of ideals alone, so we proceed to give a characterization for inverse semigroups with zero due to Baird~\cite{Baird}.

We begin with a Cuntz-Krieger uniqueness type result for inverse semigroups with zero that we will play an important role in our investigations.

\begin{Prop}
\label{p:zero-disj:faithful}
Let $S$ be an inverse semigroup with zero that is fundamental and has a $0$-disjunctive semilattice of idempotents.  Let $\varphi\colon S\to T$ be a zero-preserving homomorphism to a semigroup with zero.  Then the following are equivalent.
\begin{itemize}
\item [(1)] $\varphi$ is injective.
\item [(2)] $\varphi|_{E(S)}$ is injective.
\item [(3)] $\varphi^{-1}(0)=\{0\}$.
\item [(4)] $\varphi^{-1}(0)\cap E(S)=\{0\}$.
\end{itemize}
\end{Prop}
\begin{Proof}
Obviously, (1) implies both (2) and (3).   Also  both (2) and (3) imply (4).  It remains to show that (4) implies (1).   Suppose that (4) holds.  Then we claim that $\varphi$ is idempotent separating.  Suppose that $\varphi(e)=\varphi(f)$. If $ef=0$, then $0=\varphi(ef) = \varphi(e)\varphi(f)=\varphi(e)=\varphi(f)$ and $e=0=f$ by (4).   If $ef\neq 0$ and $ef<e$, then there exists $0\neq g\leq e$ with $gef=0$ as $E(S)$ is $0$-disjunctive.  Then $0=\varphi(gef)=\varphi(gf)=\varphi(g)\varphi(f)=\varphi(g)\varphi(e) = \varphi(ge)$ and so $g=ge=0$, a contradiction. Therefore, $ef=e$.   By symmetry, $ef=f$ and so $e=f$.  Thus $\varphi$ is idempotent separating and so $\varphi$ is injective as $S$ is fundamental.
\end{Proof}

One immediately deduces Baird's characterization of congruence-free inverse semigroups with zero (see~\cite[Theorem~IV.4.11]{petrich} for a more detailed proof).

\begin{Thm}
\label{simplesg}
An inverse semigroup $S$ with zero is congruence-free if and only if it is
fundamental, $0$-simple and $E(S)$ is $0$-disjunctive.
\end{Thm}
\begin{Proof}
If $S$ is not $0$-simple or fundamental, it obviously has a proper quotient ($S/I$ with $I$ a proper non-zero ideal for the former and $S/\mu$ for the latter).  To see that $E(S)$ is $0$-disjunctive, consider the syntactic congruence of $\{0\}$ given by $s\equiv t$ if, for all $u,v\in S$, $usv=0\iff utv=0$. Notice that
if $0\neq f<e$ and $fg\neq 0$ for all $0\neq g\leq e$, then $f\equiv e$. So if $S$ is not $0$-disjunctive, $S/{\equiv}$ is a proper, non-trivial quotient as no non-zero element is equivalent to $0$.

Conversely, if $S$ is fundamental, $0$-simple and $E(S)$ is $0$-disjunctive and $\varphi\colon S\to T$ is proper quotient, then by Proposition~\ref{p:zero-disj:faithful} we have some $0\neq e\in \varphi^{-1}(0)\cap E(S)$ and hence $S=SeS\subseteq \varphi^{-1}(0)$ by $0$-simplicity.  Thus $\varphi$ is the zero map.
\end{Proof}

\subsection{Ample groupoids}
 In this paper, following the convention popular in analysis, we shall view a groupoid $\G$ as consisting of a set of arrows; an object is identified with the identity element at that object.  The set of identities, or \emph{unit space}, is denoted $\G^0$.  If $\G$ is endowed with a topology such that the domain, range, multiplication and inversion maps are continuous, then we call $\G$ a \emph{topological groupoid}.  We always consider $\G^0$ with the subspace topology.  The topological groupoid $\G$ is
  \emph{\'etale} if its domain map $d$ (or, equivalently, its range map $r$) is a local homeomorphism.  In this case,  $\G^0$ is an open subspace and the multiplication map is a local homeomorphism.  We always consider the case where $\G^0$ is locally compact and Hausdorff; however, $\G$ need not be Hausdorff and that is crucial in this paper.  If, in addition, $\G^0$ has a basis of compact open sets, then $\G$ is called \emph{ample} following Paterson~\cite{Paterson}. Details can be found in~\cite{Paterson,resendeetale,Exel}.

A \emph{local bisection} of an \'etale groupoid $\G$ is an open subset $U\subseteq \G$ such that $d|_U$ and $r|_U$ are homeomorphisms of $U$ with $d(U)$ and $r(U)$, respectively (which are necessarily both open as these maps are local homeomorphisms).  The set $\Gamma$ of compact local bisections is a basis for the topology of an ample groupoid $\G$~\cite{Paterson,Exel} and is also an inverse semigroup under the binary operation \[UV = \{uv: u\in U,\ v\in V,\ d(u)=r(v)\}.\] The semigroup inverse is given by $U^* = \{u^{-1}: u\in U\}$ and $E(\Gamma)$ consists of the compact open subsets of $\G^0$.

The \emph{isotropy subgroupoid} of a groupoid $\G$ is the subgroupoid \[\Is(\G)=\{g\in \G : d(g)=r(g)\}.\]  An \'etale groupoid is said to be \emph{effective} if $\G^0=\mathrm{Int}(\Is(\G))$.   The \emph{isotropy group} of $\G$ at an object $x$ is the group of all arrows $g\colon x\to x$.  An \'etale groupoid is called \emph{topologically principal} if the set of objects with trivial isotropy is dense in the unit space.
 Another related notion that appears frequently in the literature, and which is equivalent to effectiveness for Hausdorff \'etale groupoids, but is different for non-Hausdorff groupoids, is termed in~\cite{KM} \emph{topological freeness}.  An \'etale groupoid $\G$ is said to be \emph{topologically free} if $\mathrm{Int}(\Is(\G)\setminus \G^0)=\emptyset$, that is, $\G^0$ is dense in $\int(\Is(\G))$.  Clearly effective groupoids are topologically free, as are topologically principal groupoids.  It is shown in~\cite{KM} via a Baire category argument  that for second countable groupoids being topologically free is equivalent to being topologically principal.  This implies the older result, cf.~\cite{nonhausdorffsimple}, that effective second countable \'etale groupoids are topologically principal.    So for second countable Hausdorff \'etale groupoids, these three notions coincide, but they are different in the generality considered in this paper.

In the work of~\cite{nonhausdorffsimple} (and the earlier related work of Nekrashevych~\cite{Nekrashevychgpd}), the property of being topologically principal played a key role in the study of simplicity and for this reason (and also because of subsequent study of $C^*$-algebras), the authors of~\cite{nonhausdorffsimple} restrict to second countable groupoids.  We shall see, however, that the property of being topologically principal is superfluous to the study of simplicity of Steinberg algebras and only effectiveness is relevant, as was the case for Hausdorff groupoids~\cite{operatorsimple1,operatorguys2,groupoidprimitive}.  In fact, if one passes to the essential algebra of an ample groupoid (to be defined later), then it is topological freeness that plays the starring role.

If $x\in \G^0$, then the \emph{orbit}  of $x$ consists of all $y\in \G^0$ such that there is an arrow $g\in \G$ with $d(g)=x$ and $r(g)=y$.  The orbits form a partition of $\G^0$.  A subset $X\subseteq \G^0$ is \emph{invariant} if it is a union of orbits.
 An \'etale groupoid is said to be \emph{minimal} if $\G^0$ has no proper, non-empty open invariant subsets or, equivalently, if each orbit is dense in $\G^0$.  The \emph{restriction} of $\G$ to an invariant subspace $X$ of $\G^0$ is the groupoid $\G|_X$ with object space $X$ and arrows all morphisms between elements of $X$ equipped with the subspace topology.  If $\G$ is ample and $X$ is either open or closed, then $\G|_X$ is also ample.

A key example of an \'etale groupoid is that of a groupoid of germs.  Let $S$ be an inverse semigroup acting on a locally compact Hausdorff space $X$ by partial homeomorphisms.  For an idempotent $e\in E(S)$, we denote by $X_e$ the domain of $e$.  We assume, to avoid degenerate situations, that $X=\bigcup_{e\in E(S)}X_e$.  The \emph{groupoid of germs $\G=S\ltimes X$} is defined as follows. One puts $\G^0=X$ and \[\G=\{(s,x)\in S\times X: x\in X_{s^*s}\}/{\sim}\] where $(s,x)\sim (t,y)$ if and only if $x=y$ and there exists $u\leq s,t$ with $x\in X_{u^*u}$. The $\sim$-class of an element $(s,x)$ is denoted $[s,x]$.  The topology on $\G$ has basis all sets of the form $(s,U)$ where $U\subseteq X_{s^*s}$ is open and $(s,U) = \{[s,x]: x\in U\}$. Each arrow $[t,x]$ has a basis of neighborhoods of the form $(t,U)$ with $U\subseteq X_{t^*t}$. The identification of $X$ with $\G^0$ sends $x\in X_e$ to $[e,x]$ (which is independent of the choice of $e$).  However, for groupoids of germs it is more convenient to view $X$ as the unit space.   One puts $d([s,x])=x$, $r([s,x])=sx$ and defines $[s,ty][t,y]=[st,y]$.  Inversion is given by $[s,x]^{-1} = [s^*,sx]$.  Note that $(s,X_{s^*s})$ is a local bisection and is compact if $X_{s^*s}$ is compact.  If $X$ has a basis of compact open sets and the $X_e$ are all compact open, then the groupoid of germs $\G$ is ample.  See~\cite{Exel,Paterson} for details.

The principal example for us is the (contracted) \emph{universal groupoid} $\G(S)$ of an inverse semigroup $S$ with zero.  (This slighlty differs from Paterson's universal groupoid~\cite{Paterson}; to obtain his one can adjoin an external zero and then do the following construction.)  The inverse semigroup $S$ acts by partial homeomorphisms on the spectrum $\widehat{E(S)}$ of its semilattice of idempotents.  The domain of $e\in E(S)$ is $D(e)$ and the partial homeomorphism $\beta_s\colon D(s^\ast s)\to D(ss^\ast)$ is given by $\beta_s(\theta)(e) = \theta(s^\ast es)$.  In particular, as the $D(e)$ are compact open, we have that the groupoid of germs $\G(S)= S\ltimes \widehat{E(S)}$ of this action is ample.  If $\Gamma$ is the inverse semigroup of compact local bisections, then there is an embedding of $S$ into $\Gamma$ sending $s$ to $(s,D(s^\ast s))$.

The space $\widehat{E(S)}_T$ of tight characters is a closed invariant subspace of $\widehat{E(S)}$ and hence we can form an ample groupoid called the \emph{tight groupoid} $\G_T(S)$ by taking $\G(S)|_{\widehat{E(S)}_T}$~\cite{Exel}.
A fundamental result of Exel~\cite{Exel} is that if $\Gamma$ is the inverse semigroup of compact local bisections of an ample groupoid $\G$, then $\G\cong \G_T(\Gamma)$.  Hence the tight groupoid construction is general enough to capture all ample groupoids.

\subsection{Boolean inverse semigroups}
Recall that two elements $s,t$ of an inverse semigroup $S$ are \emph{compatible} if $st^\ast, s^\ast t\in E(S)$.  This is a necessary condition for $s,t$ to have a join in $S$ with respect to the natural partial order. In particular, if $s,t$ are \emph{orthogonal}, meaning $st^\ast=0=s^\ast t$, then they are compatible.  Notice that inversion is an order isomorphism of an inverse semigroup and hence preserves all compatible joins.
 A Boolean inverse semigroup is an inverse semigroup $S$ such that $E(S)$ is a Boolean algebra (not necessarily unital), $S$ admits (binary) compatible joins and products distribute over joins.  We shall denote, for $e,f\in E(S)$,  the relative complement by $e\setminus f$  and the join by $e\vee f$.  Joins of compatible semigroup elements will also be denoted by $\vee$.

Because join is only partially defined on a Boolean inverse semigroup, it is convenient to introduce two totally defined operations which enable one to describe Boolean inverse semigroups as universal algebras  in an extended signature (containing the usual multiplication and inversion operations on an inverse semigroup).  This was done by Wehrung~\cite{wehrung}.  If $S$ is a Boolean inverse semigroup and $s,t\in S$, then their \emph{skew difference} is $s\ominus t = (ss^\ast\setminus tt^\ast)s(s^\ast s\setminus t^\ast t)$.      This is the largest element below $s$ and orthogonal to $t$; in particular, if $t\leq s$, then $s\ominus t$ is the relative complement of $t$ in $s$.   Also, note that $(s\ominus t)^\ast = s^\ast\ominus t^\ast$. As $s\ominus t$ and $t$ are orthogonal, we can define  \emph{skew addition} by  $s\triangledown t =(s\ominus t)\vee t$.  Wehrung shows that Boolean inverse semigroups form a variety of unary semigroups with these additional two binary operations.
In fact, the following is shown in~\cite[Theorem~3.2.5]{wehrung}.

\begin{Thm}\label{bias.hom}
Let $S$ and $T$ be Boolean inverse semigroups and $\varphi\colon S\to T$ a semigroup homomorphism.  Then the following are equivalent.
\begin{enumerate}
\item $\varphi$ preserves skew difference and skew additions;
\item $\varphi\colon E(S)\to E(T)$ is a Boolean algebra homomorphism;
\item $\varphi$ preserves binary compatible joins;
\item $\varphi$ preserves joins of orthogonal elements;
\item $\varphi$ preserves joins of orthogonal idempotents.
\end{enumerate}
\end{Thm}

A homomorphism satisfying the above equivalent properties is called \emph{additive}.  Note that Wehrung does not include (5) in his statement, but it follows from the equivalence of the rest:  if $\varphi\colon S\to T$ preserves joins of orthogonal idempotents, then $\varphi|_{E(S)}\colon E(S)\to E(T)$ is a Boolean algebra homomorphism by the equivalence of (2) and (4).

 A quotient of $S$ by a congruence preserving skew differences and skew addition is a Boolean inverse semigroup and the  quotient map is additive by the theorem.

We shall call a non-zero Boolean inverse semigroup $S$ \emph{additively congruence-free} if every non-zero additive homomorphism $\varphi\colon S\to T$ of Boolean inverse semigroups is injective.  This is equivalent to $S$ being congruence-free in the enriched signature including skew difference and skew addition.

It turns out that additively congruence-free inverse semigroups correspond precisely to effective and minimal ample groupoids under non-commutative Stone duality~\cite{LawsonLenz} via Exel's tight groupoid construction~\cite{Exel}.  First we need the notion of an additive ideal of a Boolean inverse semigroup~\cite{LawsonVdovina}.  An ideal $I$ of $S$ is \emph{additive} if $e,f\in I\cap E(S)$ and $ef=0$ implies $e\vee f\in I$.  This is equivalent to $I$ being closed under compatible joins.  The usual quotient of an inverse semigroup by an ideal (the Rees quotient) does not result in a Boolean inverse semigroup even if the ideal is additive.  For an additive ideal $S$,  define $S\idq I$ to be the quotient of $S$ by the congruence given by $s\equiv t$ if there exists $u\leq s,t$ with $s\ominus u, t\ominus u\in I$.  Then $S\idq I$ is a Boolean inverse semigroup, the quotient map $S\to S\idq I$ is additive and $s\equiv 0$ if and only if $s\in I$; see~\cite{LawsonVdovina} for details.  Let us say that a non-zero Boolean inverse semigroup $S$ is \emph{additively $0$-simple} if it has no non-zero proper additive ideals.  Note that if $\varphi\colon S\to T$ is an additive homomorphism, then $\varphi^{-1}(0)$ is an additive ideal of $S$.

We next show that if $S$ is a Boolean inverse semigroup with maximum idempotent-separating congruence $\mu$, then $S/\mu$ is a Boolean inverse semigroup and the quotient map is additive. We recall that $\mu$ is given by $s\mathrel{\mu} t$ if and only if $ses^\ast=tet^\ast$ for all $e\in E(S)$.

\begin{Prop}\label{fundquot}
Let $S$ be a Boolean inverse semigroup and $\mu$ the largest idempotent separating congruence on $S$.  Then $S/\mu$ is a Boolean inverse semigroup and the natural projection $S\to S/\mu$ is additive.
\end{Prop}
\begin{Proof}
We show that $\mu$ is compatible with skew difference and skew addition.
Let $s_1\mathrel{\mu} s_2$ and $t_1\mathrel{\mu} t_2$.  Then $s_1s_1^\ast = s_2s_2^\ast$, $s_1^\ast s_1=s_2^\ast s_2$ and $t_1t_1^\ast = t_2t_2^\ast$, $t_1^\ast t_1=t_2^\ast t_2$ as $\mu$ is idempotent separating.  It follows then that if $e\in E(S)$, then
\begin{align*}
(s_1\ominus t_1) e (s_1\ominus t_1)^* &= (s_1s_1^\ast\setminus t_1t_1^\ast)s_1(s_1^\ast s_1\setminus t_1^\ast t_1)es_1^*(s_1s_1^\ast\setminus t_1t_1^\ast)\\
(s_2\ominus t_2) e (s_2\ominus t_2)^* &= (s_2s_2^\ast\setminus t_2t_2^\ast)s_2(s_2^\ast s_2\setminus t_2^\ast t_2)es_2^*(s_2s_2^\ast\setminus t_2t_2^\ast).
\end{align*}
Since $s_1\mathrel{\mu} s_2$, we deduce that $s_1(s_1^\ast s_1\setminus t_1^\ast t_1)es_1^*=s_1(s_2^\ast s_2\setminus t_2^\ast t_2)es_1^*=s_2(s_2^\ast s_2\setminus t_2^\ast t_2)es_2^*$ and so $(s_1\ominus t_1)\mathrel{\mu} (s_2\ominus t_2)$.

To prove that $\mu$ is compatible with skew addition, it suffices to show (by what we have just proved) that if $s_1\mathrel{\mu} s_2$ and $t_1\mathrel{\mu} t_2$ with $s_1^\ast t_1=0=s_1t_1^\ast$ and $s_2^\ast t_2=0=s_2t_2^\ast$, then $(s_1\vee t_1)\mathrel{\mu} (s_2\vee t_2$).  But if $e\in E(S)$, then $(s_1\vee t_1)e(s_1\vee t_1)^\ast = s_1es_1^\ast \vee t_1et_1^\ast = s_2es_2^\ast\vee t_2et_2^\ast= (s_2\vee t_2)e(s_2\vee t_2)^\ast$ by distributivity and the orthogonality assumption.  Thus we have $(s_1\vee t_1)\mathrel{\mu} (s_2\vee t_2)$, as required.
\end{Proof}

The second part of the following proposition is due to Exel~\cite{Exel}.

\begin{Prop}
\label{Booleanalg}
Let $E$ be a Boolean algebra.
\begin{enumerate}
\item $E$ is $0$-disjunctive.
\item $e_1,\ldots, e_n\leq e$ cover $e$ if and only if $e=e_1\vee\cdots \vee e_n$.
\end{enumerate}
\end{Prop}
\begin{Proof}
Let $0\neq e$ be an idempotent and $0\neq f<e$.  Then $0\neq e\setminus f$ and $f(e\setminus f)=0$.  Thus $E$ is $0$-disjunctive.  Clearly, if $e=e_1\vee\cdots \vee e_n$, then $e_1,\ldots, e_n$ cover $e$ as $fe=fe_1\vee\cdots \vee fe_n$ for $f\in E$.  Conversely, if $e_1,\ldots, e_n$ is a cover of $e$ and $f=e_1\vee\cdots \vee e_n$, then $e\setminus f$ is orthogonal to $e_1,\ldots, e_n$ and hence $e\setminus f=0$ by definition of a cover.  Thus $e=f$.
\end{Proof}

We remark that it follows from Proposition~\ref{Booleanalg} that a filter on a Boolean algebra is tight if and only if it is an ultrafilter.

\begin{Thm}
\label{addcongfree}
A non-zero Boolean inverse semigroup is additively congruence-free if and only if it is fundamental and additively $0$-simple.
\end{Thm}
\begin{Proof}
If $S$ is additively congruence-free, then since $S\to S/\mu$ is additive by Proposition~\ref{fundquot}, and non-zero, we must have that $S$ is fundamental.  If $0\neq I$ is an additive ideal of $S$, then $S\to S\idq I$ is additive and $s\equiv 0$ if and only if $s\in I$.  Thus $I=0$ or $I=S$ and so $S$ is additively $0$-simple.  Suppose that $S$ is fundamental and additively $0$-simple and let $\varphi\colon S\to T$ be a non-zero additive homomorphism of Boolean inverse semigroups.  Then $\varphi^{-1}(0)$ is an additive ideal of $S$ and hence must be $0$.  Thus $\varphi$ is injective by Proposition~\ref{p:zero-disj:faithful}, as $E(S)$ is $0$-disjunctive by Proposition~\ref{Booleanalg} and $S$ is fundamental.
\end{Proof}

For any inverse semigroup $S$, if $F\subseteq S$, we write $F^{\downarrow}$ for the order ideal generated by $F$; so $F^{\downarrow}$ consists of all $s\in S$ with $s\leq t$ for some $t\in F$.  We use $s^{\downarrow}$ as shorthand for $\{s\}^{\downarrow}$.
An inverse semigroup is called \emph{Hausdorff} if, for any two $s,t\in S$, the set $s^{\downarrow}\cap t^{\downarrow}$ is a finitely generated order ideal; that is, $s^{\downarrow}\cap t^{\downarrow}=F^{\downarrow}$ for some finite $F\subseteq S$. This is equivalent to the intersection of finitely generated order ideals being finitely generated.  Note that a Boolean inverse semigroup is Hausdorff if and only if it has meets with respect to the natural partial order since the elements of $s^{\downarrow}\cap t^{\downarrow}$ are compatible and hence this set is closed under finite joins; the join of a finite generating set for this order ideal gives a meet.  It was shown in~\cite{mygroupoidalgebra} that $S$ is Hausdorff if and only if the groupoid $\G(S)$ is Hausdorff, hence the term.
We note that for Hausdorff Boolean inverse semigroups, the notions of fundamental and quasi-fundamental coincide.

\begin{Prop}\label{p:qv=fund.hs}
A Hausdorff Boolean inverse semigroup $S$ is quasi-fundamental if and only if it is fundamental.
\end{Prop}
\begin{Proof}
It suffices to show that if $S$ is quasi-fundamental, then $S$ is fundamental.  Let $0\neq s\in Z_S(E(S))$.  As remarked above, $s$ and $s^\ast s$ have a meet $e\in E(S)$ because $S$ is a Hausdorff Boolean inverse semigroup.  Then $\{f \in E(S): f \leq s\}=e^{\downarrow}$.  Consider $t=s(s^\ast s\setminus e)$.  Then $t\leq s$ and $t\in Z_S(E(S))$.  If $t\neq 0$, then there exists $0\neq f\leq t\leq s$ an idempotent because $S$ is quasi-fundamental.  But then $f\leq e$, which is impossible as $f=tf=tef=0$.  Thus $t=0$ and so $s^\ast s\setminus e=0$, that is, $s^\ast s=e\leq s$.  But then $s=ss^\ast s=s^\ast s$, and so $s\in E(S)$.  Thus $S$ is fundamental.
\end{Proof}

The assumption that $S$ is Boolean in Proposition~\ref{p:qv=fund.hs} can be weakened to $S$ having binary meets and $E(S)$ being $0$-disjunctive without changing the conclusion.  But it does not hold for arbitrary Hausdorff inverse semigroups, or even those possessing binary meets. For example if $G$ is a group, then the inverse semigroup $S=G\cup \{z,0\}$ with $zG=\{z\}=Gz$ and $z^2=z$, constructed earlier, is quasi-fundamental but not fundamental and it is Hausdorff since it has binary meets.

If $\G$ is an ample groupoid, let $\Gamma$ be the inverse semigroup of compact local bisections of $\G$; it is well known and easy to check that $\Gamma$ is a Boolean inverse semigroup, cf.~\cite{LawsonLenz}. Joins are given by set-theoretic union. As mentioned earlier, Exel proved that
$\G\cong \G_T(\Gamma)$~\cite{Exel}. Conversely, for any Boolean inverse semigroup $S$, its tight groupoid (which is the same as its ultrafilter groupoid) $\G_T(S)$ is an ample groupoid and $S$ is isomorphic to the Boolean inverse semigroup of compact local bisections of $\G_T(S)$ via $s\mapsto (s,D(s^\ast s)\cap \widehat{E(S)}_T)$; cf.~\cite{LawsonLenz}.
 We now show that under this dictionary between ample groupoids and Boolean inverse semigroups, the notions of fundamental and additively $0$-simple correspond to effectiveness and minimality, respectively.  We also show that quasi-fundamental corresponds to topologically free.

The following observation is essentially folklore, but we include a proof for completeness.
\begin{Prop}\label{p:centralizerofidems}
Let $\G$ be an ample groupoid  with inverse semigroup of compact local bisections $\Gamma$.  Let $\Gamma'$ be the inverse semigroup of compact local bisections of the ample groupoid $\int(\Is(\G))$.    Then $Z_{\Gamma}(E(\Gamma)) = \Gamma'$.
\end{Prop}
\begin{Proof}
Suppose first that $U\in \Gamma'$. Now, given such a $U$,  if $V\in E(\Gamma)$, then $V$ is a compact open
subset of the unit space and so $VU= \{g\in U: r(g)\in V\}=\{g\in U: d(g)\in V\}=UV$ as $d(g)=r(g)$ for all $g\in U$.  Thus $U\in Z_{\Gamma}(E(\Gamma))$.   Conversely, assume that $U\in \Gamma$ centralizes $E(\Gamma)$.  Suppose that there exists $g\in U$ with $d(g)\neq r(g)$.  Then since the unit space of $\G$ is Hausdorff, we can find $V\in E(\Gamma)$ with $d(g)\in V$ and $r(g)\notin V$.  Then $g\in UV$ and $g\notin VU$, contradicting $UV=VU$.  Thus $U\in \Gamma'$.  This completes the proof.
\end{Proof}

\begin{Prop}
\label{translation}
Let $\G$ be an ample groupoid and $\Gamma$ its inverse semigroup of compact local bisections.
\begin{itemize}
\item [(1)] $\G$ is effective if and only if $\Gamma$ is fundamental;
\item [(2)] $\G$ is minimal if and only if $\Gamma$ is additively $0$-simple;
\item [(3)] $\G$ is topologically free if and only if $\Gamma$ is quasi-fundamental.
\end{itemize}
\end{Prop}
\begin{Proof}
Since a basis for the interior of the isotropy subgroupoid $\mathrm{Iso}(\G)$ is the set of $U\in \Gamma$ with $U\subseteq \mathrm{Iso}(\G)$, we see that $\G$ is effective if and only if $\Gamma'=E(\Gamma)$, where $\Gamma'$ is the inverse semigroup of compact local bisections of $\int(\Is(\G))$.  We may then deduce (1) from Proposition~\ref{p:centralizerofidems}.

It is shown in~\cite{LawsonVdovina} that the lattice of additive ideals of $\Gamma$ is isomorphic to the lattice of open invariant subsets of the unit space of $\G$.  Thus $\Gamma$ is additively $0$-simple if and only if $\G$ is minimal.  The correspondence sends an additive ideal $I$ to $\bigcup_{U\in I\cap E(\Gamma)} U$ and takes an open invariant subset $W$ of the unit space to \[\bigcup\limits_{U\in E(\Gamma), U\subseteq W}\Gamma U\Gamma\]  (this bijection is folklore and can also be essentially found in~\cite{LalondeMilanK} for tight groupoids of inverse semigroups).

If $\G$ is topologically free and $0\neq U\in Z_\Gamma(E(\Gamma))$, then by Proposition~\ref{p:centralizerofidems} we have
that $U\subseteq \int(\Is(\G))$ and hence $U\cap \G^0\neq\emptyset$ because $\G$ is topologically free.  Since $\G^0$ is open, as is $U$, we can find a compact open subset $V\neq \emptyset$ with $V\subseteq U\cap \G^0$ and then in $\Gamma$ we have $0\neq V\leq U$ with $V\in E(\Gamma)$.  Thus $\Gamma$ is quasi-fundamental.  Conversely, if  $\G$ is not topologically free, we can find $g\in \int(\Is(\G)\setminus \G^0)$.  Then we can find a compact local bisection $V$ with $g\in V\subseteq \Is(\G)\setminus \G^0$.
 We deduce from Proposition~\ref{p:centralizerofidems} that $V\in Z_\Gamma(E(\Gamma))$.  Since the idempotents of $\Gamma$ are the compact open subsets of $\G^0$ and the natural partial order on $\Gamma$ is inclusion, we deduce from $V\cap \G^0=\emptyset$ that there is no non-zero idempotent below $V$ and hence $\Gamma$ is not quasi-fundamental.  This completes the proof of (3).
\end{Proof}

\subsection{Inverse semigroup and ample groupoids algebras}
If $S$ is an inverse semigroup and $K$ is a commutative ring with unit, then the semigroup algebra $KS$ is the $K$-algebra with basis $S$ and product extending that of $S$.  For semigroups with zero, it is convenient to identify the zeroes of $S$ and of $K$.  So the \emph{contracted semigroup algebra} $\K S$ of $S$ is the $K$-algebra with basis $\S=S\setminus \{0\}$ and multiplication extending that of $S$ (where the zero of $S$ and $\K S$ are the same).  Of course, $KS = \K[S^0]$, where $S^0$ is the result of adjoining an external zero to $S$, and so there is no real loss of generality sticking with contracted semigroup algebras.  Note that the natural map $S\to \K S$ is the universal zero-preserving homomorphism of $S$ into a $K$-algebra.

If $K$ is a commutative ring with unit and $\G$ is an ample groupoid, the so-called \emph{Steinberg algebra}~\cite{mygroupoidalgebra} $K\G$ is the $K$-span of the characteristic functions of compact local bisections (in $K^{\G}$) with the convolution product
\[f\ast g(\gamma) = \sum_{\alpha\beta=\gamma}f(\alpha)g(\beta).\]  The $K$-algebra $K\G$ is associative, but is unital if and only if $\G^0$ is compact.  If $\Gamma$ is the inverse semigroup of compact local bisections, then $\chi_U\ast\chi_V = \chi_{UV}$ for $U,V\in \Gamma$ (where $\chi_X$ denotes the characteristic function of a set $X$).  Thus there is a surjective $K$-algebra homomorphism $\Psi_K\colon \K \Gamma\to K\G$.  The kernel of this map is generated as an ideal by all $U+V-(U\cup V)$ with $U,V$ disjoint compact open subsets of $\G^0$.  This was proved in~\cite{mygroupoidalgebra} when $\G$ is Hausdorff and in the general case by A.~Buss and R.~Meyer (private communication).  We provide a proof in Corollary~\ref{c:grpdpres} as a consequence of a more general result.

  If $\G$ is Hausdorff, then $K\G$ consists of the locally constant $K$-valued functions on $S$ with compact support.  In particular, the support of any function is compact open in the Hausdorff case.  The situation for non-Hausdorff groupoids is quite different.  We are primarily interested in the case where $K$ is a field here, but sometimes we shall need to deal with more general rings.  The following theorem was the initial motivation for introducing Steinberg algebras and explains why the study of simplicity of Steinberg algebras and contracted inverse semigroup algebras are interconnected; see~\cite[Theorem~5.1]{groupoidprimitive}, although the analogue for inverse semigroup algebras was proved earlier in~\cite{mygroupoidalgebra}.

\begin{Thm}
\label{isom}
If $S$ is an inverse semigroup with zero and $\G(S)$ is the universal groupoid of $S$, then $\K S$ and $\stein S$ are isomorphic via the linear map which sends $s\in \S$ to $\chi_{(s, D(s^\ast s))}$.
\end{Thm}

The following reformulation of Theorem~\ref{isom} will be useful.

\begin{Cor}
\label{isomnice}
For any $a=\sum\limits_{s \in \S}a_ss \in \K S$, let
\[f_a \colon \G(S) \to K,\quad  [t, \varphi]\mapsto \sum\limits_{[u, \varphi]=[t,\varphi]}a_u.\]
Then the map $\Phi\colon \K S \to\stein S, a \mapsto f_a$ is an isomorphism.
\end{Cor}
\begin{Proof}
Notice that $\Phi$ is linear, so by Theorem~\ref{isom}, it suffices to check that $\Phi(s)=\chi_{(s, D(s^\ast s))}$ for any $s \in \S$. By definition,
$$f_s([t, \varphi])=
\begin{cases}
1, \hbox{ if }[s, \varphi]=[t, \varphi]\\
0, \hbox{ otherwise, }
\end{cases}$$
and we have $(s, D(s^\ast s))=\{[s, \varphi] \in \mathcal G(S): \varphi \in D(s^\ast s)\}=\{[t, \varphi] \in \G(S): [t, \varphi]=[s, \varphi]\}$. So $f_s=\chi_{(s, D(s^\ast s))}$ indeed.
\end{Proof}

Recall that the idempotents of any semigroup (and hence any ring) are partially ordered by putting $e\leq f$ if $ef=e=fe$.
In any ring $R$, given two commuting idempotents $e,f$, they have a join $e\vee f=e+f-ef$ and a relative complement $e\setminus f=e-ef$ in the natural partial order on idempotents.  They also have meet $ef$.  Notice that if $ef=0=fe$, i.e., the idempotents are orthogonal, then $e\vee f=e+f$.  More generally, any set of pairwise commuting idempotents in a ring generates a Boolean algebra under these operations.  In particular, in $KS$ (or any quotient of $KS$), we have that $E(S)$ (or its image) generates a Boolean algebra.

Let $\G_T(S)$ be the tight groupoid of an inverse semigroup.  We wish to give a generating set for the kernel of the natural homomorphism $\K S\to \tight S$ as an ideal.  This was done in~\cite[Corollary 5.3]{groupoidprimitive} for the Hausdorff case;  we handle here the general case.

Let $\G$ be an ample groupoid and $X$ a closed invariant subspace of the unit space with open invariant complement $X^c$.  Then  $K\G|_{X^c}$ is an ideal in $K\G$ with $K\G/K\G|_{X^c}\cong K\G|_X$; more precisely, restricting an element $f\in K\G$ to $\G|_X$ gives a valid element of $K\G|_X$, the restriction homomorphism is onto and the kernel of the restriction homomorphism is $K\G|_{X^c}$ (this is not obvious in the non-Hausdorff case).  A proof is given in the discussion following Proposition~5.3 of~\cite{gcrccr}.

\begin{Prop}\label{p:presentation}
Let $S$ be an inverse semigroup, $X$ a closed invariant subspace of the unit space of its universal groupoid $\G(S)$ and $K$ a commutative ring with unit. Then \[K\G(S)|_X\cong KS/\left(\prod_{i=1}^n(e-e_i)\mid e_i\leq e, D(e)\cap D(e_1)^c\cap\cdots \cap D(e_n)^c\cap X=\emptyset\right)\] (where possibly $n=0$).
\end{Prop}
\begin{Proof}
First note that if $U$ is a compact local bisection in $\G(S)|_{X^c}$, then $\chi_U = \chi_U\ast \chi_{U^\ast U}$ and so $I=K\G(S)|_{X^c}$ is generated as an ideal by the characteristic functions of compact open subsets of $X^c$.
If $U\subseteq X^c$ is compact open, then $U=\bigcup_{i=1}^nB_i$ where $B_i$ are basic compact open subsets of $\widehat {E(S)}$ (necessarily satisfying $B_i\cap X=\emptyset$) and hence $\chi_U=\bigvee_{i=1}^n \chi_{B_i}$.  We deduce that $I$ is generated by the characteristic functions $\chi_B$ of basic compact open subsets $B$ of $X^c$, as a join can be expressed as an alternating sum of products via the principle of inclusion-exclusion.  Such a basic neighborhood $B$ is of the form $B=D(e)\cap D(e_1)^c\cap\cdots \cap D(e_n)^c$ where $e_1,\ldots, e_n\leq e$ with $B\cap X=\emptyset$.  Then $\chi_B=\prod_{i=1}^n (\chi_{D(e)}-\chi_{D(e_i)})$.  Under the isomorphism $KS\to K\G (S)$, we have that $\chi_B$ is the image of $\prod_{i=1}^n(e-e_i)$ and so the proposition follows.
\end{Proof}

As a corollary, we obtain the following result.

\begin{Cor}\label{c:tightpres}
Let $S$ be an inverse semigroup with zero and let $K$ be a commutative ring with unit.  Then
\begin{align*}
\tight S &\cong \K S/\left(e-\bigvee F\mid F\ \text{covers}\ e\right)\\
&\cong \K S/\left(\prod_{f\in F}(e-f)\mid F\ \text{covers}\ e\right).
\end{align*}
Hence the natural map $S\to \tight S$  is the universal homomorphism from $S$ into a $K$-algebra sending covers to joins.
\end{Cor}
\begin{Proof}
Let $X=\widehat{E(S)}_T$ be the tight spectrum of $E(S)$.  We just need to show that $D(e)\cap D(e_1)^c\cap\cdots\cap D(e_n)^c\cap X=\emptyset$ if and only if $e_1,\ldots, e_n$ cover $e$, where possibly $n=0$, in which case $e=0$.
If $e_1,\ldots, e_n$  is a cover of $e$, then $D(e)\cap D(e_1)^c\cap\cdots\cap D(e_n)^c$ cannot contain a tight character.  Conversely, if $e_1,\ldots, e_n$ is not a cover of $e$, then there exists $z$ with $0\neq z\leq e$ and $ze_i=0$ for $i=1,\ldots,n$.  Let $\mathcal F$ be an ultrafilter containing $z$.   Then $e\in \mathcal F$ and $e_1,\ldots, e_n\notin \mathcal F$.  Thus $\chi_{\mathcal F}\in D(e)\cap D(e_1)^c\cap\cdots\cap D(e_n)^c\cap X$ because the characteristic function of an ultrafilter is a tight character.  Since the empty set covers $0$, the result now follows from Proposition~\ref{p:presentation}.
\end{Proof}

As every ample groupoid is isomorphic to one of the form $\G_T(S)$, the above result lets us apply inverse semigroup theory to study groupoid algebras.  We call the ideal of $\K S$ in Corollary~\ref{c:tightpres} the \emph{tight ideal} of $\K S$ and denote it by $\tightid{K}{S}$.  We denote by $\eta$ the canonical homomorphism $\eta\colon S\to \K S/\tightid{K}{S}$.  We shall most of the time identify $\K S/\tightid{K}{S}$ with $\tight S$, but it is important to note that as a function in $\tight S$, we have that $\eta(s) = \chi_{(s,D(s^\ast s)\cap \widehat{E(S)}_T)}$.

Applying Corollary~\ref{c:tightpres} to the case of Boolean inverse semigroups yields:

\begin{Cor}[Buss-Meyer]\label{c:grpdpres}
Let $\G$ be an ample groupoid and $K$ a commutative ring with unit. Let $\Gamma$ be the inverse semigroup of compact local bisections of $\G$.  Then \[K\G\cong \K \Gamma/(U+V-(U\cup V):U\cap V=\emptyset, U,V\in E(\Gamma)).\]
\end{Cor}
\begin{Proof}
Let $I$ be the ideal generated by all $U+V-(U\cup V)$ with $U,V$ compact open disjoint subsets of the unit space of $\G$.
The isomorphism $\mathcal G \to \mathcal G_T(\Gamma)$ induces an isomorphism $K\G \to \tight \Gamma$ sending $\chi_U$ to $\chi_{(U,D(U^\ast U)\cap \widehat{E(\Gamma)}_T)}$, so it suffices to show that $I=\tightid{K}{\Gamma}$. It follows from Proposition~\ref{Booleanalg} and Corollary~\ref{c:tightpres} that $I\subseteq \tightid{K}{\Gamma}$.  However, the natural map $E(\Gamma)\to \K\Gamma/I$ preserves joins of orthogonal idempotents by the definition of $I$ and hence preserves all joins by Theorem~\ref{bias.hom} (applied to the Boolean algebra $E(\Gamma)$ and the Boolean algebra in $\K\Gamma/I$ generated by the image of $E(\Gamma)$).  Thus $\tightid{K}{\Gamma}\subseteq I$ by Proposition~\ref{Booleanalg} and Corollary~\ref{c:tightpres}. This completes the proof.
\end{Proof}

\section{Simplicity of inverse semigroup algebras and the singular ideal}
\label{s:simpsemialg}
In this section we introduce the algebraic analogue of the ideal of singular functions in the context of contracted inverse semigroup algebras.  We prove a uniqueness theorem for the quotient by the singular ideal for quasi-fundamental inverse semigroups with $0$-disjunctive semilattice of idempotents and use the uniqueness theorem to describe simple contracted inverse semigroup algebras, proving Theorem~B.

\subsection{The singular ideal}
In this subsection assume that $K$ is a commutative ring with unit.
If $a=\sum_{s\in \S}a_ss\in \K S$, then by the support of $a$ we mean the set $\supp(a)$ of $s\in \S$ with $a_s\neq 0$; this is a finite set.

\begin{Prop}
\label{singdef}
Let $S$ be an inverse semigroup with zero, and $\K S$ its contracted algebra. Let $a=\sum\limits_{s \in \S}a_ss$ be an arbitrary element of $\K S$. Then the following are equivalent:
\begin{enumerate}
\item[\namedlabel{singdef:symm}{(S)}] for all $t \in \S$, there exists $0 \neq u \leq t$ such that $\sum_{s \geq u} a_s=0$;
\item[\namedlabel{singdef:right}{(R)}] for all $e \in E\setminus \o$, there exists $0 \neq f \leq e$ such that $af=0$;
\item[\namedlabel{singdef:left}{(L)}] for all $e \in E\setminus \o$, there exists $0 \neq f \leq e$ such that $fa=0$.
\end{enumerate}
\end{Prop}

We will call such an element $a$ \emph{singular} in the sequel.  Note that $0$ is singular.

\begin{Proof}
We prove that~\ref{singdef:symm} and~\ref{singdef:right} are equivalent, the equivalence of~\ref{singdef:symm} and~\ref{singdef:left} follows from a symmetric argument.

Suppose $a$ satisfies~\ref{singdef:right}, and let $t \in \S$. Choose $0 \neq f \leq t^\ast t$ with $af=0$, and put $u=tf$. We prove that $\sum_{s \geq u} a_s=0$.
The condition $s \geq u$ is equivalent to $u=sf$:
if $sf=u$, then clearly $s \geq u$, and conversely if $s \geq u$, then $u=su^\ast u=sft^\ast tf=sf$. Hence the coefficient of $u$ in $af$ is \[\sum\limits_{u=sf} a_s=\sum\limits_{s \geq u} a_s,\] which is of course $0$. So $a$ satisfies~\ref{singdef:symm} indeed.

The converse implication~\ref{singdef:symm} $\Longrightarrow$~\ref{singdef:right} we prove by contradiction.
 Suppose that $a=\sum_{s \in \S}a_ss$ satisfies~\ref{singdef:symm} but not~\ref{singdef:right}, furthermore suppose $a$ has support of minimal cardinality with respect to this property. Note that $a \neq 0$.

We claim that, for any idempotent $e$, $ae$ still satisfies~\ref{singdef:symm}. Put
\[ae=\sum\limits_{s \in \S} a_s(se)=\sum\limits_{v \in \S} b_vv,\] where $b_v=\sum_{se=v}a_s$, and let $t \in \S$. We need $0 \neq u \leq t$ with $\sum_{v \geq u} b_v=0$.
If $te=0$, then $\sum_{v \geq t}b_v=\sum_{se \geq t}a_s=\sum_{s \in \emptyset}a_s=0$, as
 $se \geq t$ would imply $t=set^\ast t=st^\ast te=0$. Hence in this case $ae$ satisfies~\ref{singdef:symm} with $u=t$.

Otherwise, choose $0 \neq u \leq te$ such that  $\sum_{s \geq u}a_s=0$. We claim that $s \geq u$ if and only if $se \geq u$, that is, $su^\ast u=u$ if and only if $seu^\ast u=u$:
notice that, as $u \leq te$, we have $u=ue$ and so $u^\ast u=u^\ast ue=e u^\ast u$, so for any $s \in S$ we have $su^\ast u=se u^\ast u$. This implies
\[0=\sum\limits_{s \geq u}a_s=\sum\limits_{se \geq u}a_s=\sum\limits_{v \geq u}b_v,\]
that is, $ae$ satisfies~\ref{singdef:symm} for any idempotent $e$ indeed.

Since $a$ does not satisfy~\ref{singdef:right}, there exists some idempotent $e \in E\setminus \o$ such that for all $0 \neq f \leq e$, $af \neq 0$. Then $ae$ also does not satisfy~\ref{singdef:right} on account of $e$ (as $aef=af$, for all $f\leq e$) but satisfies~\ref{singdef:symm}, and $|\supp(ae)|\leq |\supp(a)|$. Therefore, by replacing $a$ by $ae$, we may without loss of generality assume $ae=a$ holds for our counterexample, that is, whenever $t \in \supp(a)$, we have $t^\ast t \leq e$ and $af\neq 0$ for all $0\neq f\leq e$.

Let $t \in \supp(a)$, and choose $0 \neq u \leq t$ such that $\sum_{s \geq u}a_s=0$. Notice that $0 \neq u^\ast u \leq t^\ast t \leq e$ and so $au^\ast u\neq 0$.  Again, $au^\ast u$ does not satisfy~\ref{singdef:right} on  account of $u^\ast u$ (since if $0\neq f\leq u^\ast u$, then $f\leq e$ and so $au^\ast uf=af\neq 0$) but satisfies~\ref{singdef:symm} by our above claim. We claim that $|\supp(au^\ast u)|<|\supp(a)|$, which would contradict the minimality of $\supp(a)$. In particular, the coefficient of $u=tu^\ast u$ is
\[\sum\limits_{su^\ast u=u}a_s=\sum\limits_{s\geq u}a_s=0.\]
Since $\supp(au^\ast u)\subseteq \supp(a)u^\ast u$, we obtain $|\supp(au^\ast u)|<|\supp(a)|$,  yielding a contradiction.
\end{Proof}

\begin{Thm}
The set of all singular elements forms an ideal of $\K S$.
\end{Thm}

\begin{Proof}
Suppose that $a$, $b$ are singular. Let $0\neq e \in E$ and choose $0 \neq f_a \leq e$ such that $af_a=0$ and $0 \neq f_b \leq f_a$ such that $bf_b=0$. Then $(a+b)f_b=af_af_b+bf_b=0$, and so $a+b$ satisfies~\ref{singdef:right} and is therefore singular.
Furthermore for any $c \in \K S$, it is immediate that $ac$ satisfies~\ref{singdef:left} and $ca$ satisfies~\ref{singdef:right}, which proves the statement.
\end{Proof}

We call the ideal of singular elements the \emph{singular ideal} of $\K S$. The quotient of $\K S$ by it singular ideal will be called the \emph{essential algebra} of $S$ over $K$ in analogy with the $C^*$-algebra terminology in~\cite{EP19,KM}.

\begin{Prop}
\label{tight}
The singular ideal contains the tight ideal $\tightid{K}{S}$.
\end{Prop}
\begin{Proof}
 If $e, f_1, \ldots, f_n \in E\setminus \o$ are such that $f_1, \ldots, f_n$ cover $e$, then $e-\bigvee_{i} f_i$ is always singular by the following argument. For any $0\neq h \in E$, if $eh=0$, then $(e-\bigvee_{i} f_i)h=0$; otherwise, $eh\leq e$ and so there exists some index $j$ such that $f_jh=f_jeh\neq 0$.  Thus $(e-\bigvee_{i} f_i)f_jh=f_jh-f_jh=0$.  Corollary~\ref{c:tightpres} then yields the statement.
\end{Proof}

 We shall later need the following observation.

\begin{Prop}
\label{embed:singular.quotient}
Let $S$ be an inverse semigroup with zero and $I$ the singular ideal of $\K S$.  Then $I$ contains no element of $\S$; in particular, the singular ideal is proper.
\end{Prop}
\begin{Proof}
Elements of $\S$ are never singular in $\K S$ as $se\neq 0$ for all $0\neq e\leq s^\ast s$.
\end{Proof}

Recall that an inverse semigroup is called Hausdorff if, for any two $s,t\in S$, the set $s^{\downarrow}\cap t^{\downarrow}$ is a finitely generated order ideal. For Hausdorff inverse semigroups, the tight ideal and the singular ideal coincide.  This can be proved topologically using a topological characterization of the singular ideal, but we prefer a direct algebraic proof.

\begin{Prop}
\label{tightissingular:hauss}
Let $S$ be a Hausdorff inverse semigroup with zero.  Then the singular ideal and the tight ideal of $\K S$ coincide for any commutative ring with unit $K$.
\end{Prop}
\begin{Proof}
We already know that the tight ideal $\tightid{K}{S}$ is contained in the singular ideal $I$ by Proposition~\ref{tight}.  For the converse, we prove by induction on $|\supp(a)|$ that if $a\in I$, then $a\in \tightid{K}{S}$.  This is obvious if $a=0$, i.e., $\supp(a)=\emptyset$.  Suppose that it is true for elements of $I$ of smaller cardinality support than $a\in I$.  Since $S$ is Hausdorff, for each $s\neq t\in S$, we can find a finite subset $F_{s,t}$ with $s^{\downarrow}\cap t^{\downarrow} = F_{s,t}^{\downarrow}$.  Let \[F=\bigcup_{s, t\in \supp(a), s\neq t}\{u^\ast u: u\in F_{s,t}\}.\]  We claim that $s^*sF$ is a cover of $s^*s$ for all $s\in \supp(a)$; in particular, $F$ is non-empty.  Let $0\neq e\leq s^\ast s$. We need $e'\in s^\ast sF$ such that $ee'\neq 0$.  Since $a\in I$, we can find $0\neq f\leq e$ with $af=0$.  Then $sf\neq 0$, but $\sum_{tf=sf}a_t=0$ and so there is some $t\neq s$ in $\supp(a)$ with $tf=sf$.  Then $sf=tf\leq s,t$ and so $sf\leq u$ for some $u\in F_{s,t}$ and hence $f=(sf)^\ast (sf)\leq u^\ast u=s^\ast su^\ast u\in s^\ast sF$.  Thus $u^\ast uf\neq 0$ (whence, $u^\ast ue\neq 0$) and so $e'=u^\ast u\in s^\ast sF$ is not orthogonal to $e$, establishing that $s^\ast sF$ is a cover of $s^\ast s$.  We deduce that $s(s^\ast s-s^\ast s(\bigvee F))=s-s(\bigvee F)\in \tightid{K}{S}$, i.e, $s+\tightid{K}{S}=s(\bigvee F)+\tightid{K}{S}$.  As $s\in \supp(a)$ was arbitrary, we have that $a+\tightid{K}{S}= a(\bigvee F)+\tightid{K}{S}$.  Moreover, $a(\bigvee F)$ is a finite sum of terms of the form $\pm ae_1\cdots e_k$ where $e_1,\ldots, e_k\in F$.  Such a term belongs to $I$.  We claim that $|\supp(ae_1\cdots e_k)|<|\supp(a)|$ for all $e_1,\ldots, e_k\in F$.  It will then follow by induction that $a(\bigvee F)\in \tightid{K}{S}$ and hence $a\in \tightid{K}{S}$.  Indeed, there are $s\neq t$ with $s,t\in \supp(a)$ and with $e_1=u^\ast u$ for some $u\in F_{s,t}$.  Then $se_1=u=te_1$ and so $|\supp(ae_1\cdots e_k)|\leq |\supp(ae_1)|\leq |\supp(a)e_1|< |\supp(a)|$.  This completes the proof.
\end{Proof}

\subsection{Simplicity and simple quotients}
We now assume that $K$ is a field.
If $a=\sum_{s \in \S}a_ss$ is a non-singular element of $\K S$, then it fails~\ref{singdef:symm} and so there exists $t \in \S$ such that, for all $0 \neq u \leq t$, we have $\sum_{s \geq u}a_s \neq 0$. We call such $t$ \emph{magic} for $a$.

\begin{Prop}
\label{magic}
If $a$ is non-singular with magic element $t$, and $e,f \in E\cup \{1\}$ (where $1$ is an adjoined identity) are such that $etf\neq 0$, then $eaf$ is non-singular with magic element $etf$.
\end{Prop}

\begin{Proof}
Let $0\neq u \leq etf$.
We claim that, for any $s$, we have $s \geq u$ if and only if $esf \geq u$: indeed,
if $s \geq u$, then $esf \geq euf =u$, and if  $esf \geq u$, then $s \geq esf$ yields $s \geq u$. So if $eaf=\sum\limits_{s \in \S}a_sesf=\sum\limits_{v \in \S}b_v v$, then
$$\sum\limits_{v \geq u}b_v=\sum\limits_{esf \geq u}a_s=\sum\limits_{s \geq u}a_s \neq 0,$$
as $0 \neq u \leq t$ and $t$ was magic for $a$.
Thus $etf$ is magic for $eaf$.
\end{Proof}

In order to prove our two main results on simplicity of contracted inverse semigroup algebras and Steinberg algebras, we prove an inverse semigroup analogue of the Cuntz-Krieger uniqueness theorem.

\begin{Thm}[Uniqueness]
\label{uniqueness}
Let $S$ be a quasi-fundamental inverse semigroup with zero whose idempotent semilattice is $0$-disjunctive and $K$ a field.  Let $I$ be the singular ideal of $\K S$.  Then the following are equivalent for an ideal $J\lhd \K S$.
\begin{itemize}
\item [(1)]  $J\cap S=\{0\}$.
\item [(2)]  $J\cap E(S)=\{0\}$.
\item [(3)]  $J\subseteq I$.
\end{itemize}
Moreover, if $S$ is fundamental, then these are equivalent to:
\begin{itemize}
\item [(4)] The canonical mapping $S\to \K S/J$ is injective;
\item [(5)]  The canonical mapping $E(S)\to \K S/J$ is injective;
\end{itemize}
and, consequently, $I$ is the unique largest ideal $J$ of $\K S$ such that the canonical mapping $S\to \K S/J$ is injective.
\end{Thm}
\begin{Proof}
The equivalence of (1) and (2) is trivial since $s=ss^\ast s$ and $s^\ast s$ is idempotent.
 On the other hand, (3) implies (1) by Proposition~\ref{embed:singular.quotient}.  Proposition~\ref{p:zero-disj:faithful} shows the equivalence of (1), (2), (4) and (5) in the case of fundamental inverse semigroups.  It remains to show that (2) implies (3).

We will prove by induction on $|\supp(a)|$ that, for any $a \in \K S \setminus I$, we have $(a)\cap E(S)\neq 0$, where $(a)$ denotes the two-sided ideal generated by $a$. If $|\supp(a)|=1$, then $a=a_ss$ with $s\in \S$ and $a_s\neq 0$ and so $0\neq s^\ast s=\frac{1}{a_s} s^\ast a\in (a)\cap E(S)$.

Suppose that $|\supp(a)|\geq 2$, and assume that the claim holds for all elements of $\K S \setminus I$ with smaller cardinality support.
Choose a magic element $t \in S$ for $a$. Since $(tt^\ast at^\ast t) \subseteq (a)$, it suffices to prove $(tt^\ast at^\ast t)\cap E(S)\neq 0$. By Proposition~\ref{magic},  $tt^\ast at^\ast t$ is, too, non-singular with magic element $tt^\ast tt^\ast t=t$. If $|\supp(tt^\ast at^\ast t)| <|\supp(a)|$, then this completes the proof by induction. Otherwise, by replacing $a$ with $tt^\ast at^\ast t$, we may assume without loss of generality that $a=tt^\ast at^\ast t$, that is, assume that, for all $s \in \supp(a)$, we have $s^ \ast s \leq t^\ast t$ and $ss^\ast \leq tt^\ast$.

Suppose $s^\ast s < t^\ast t$ for some $s \in \supp(a)$. Then since $E(S)$ is $0$-disjunctive, there exists $0 \neq f < t^\ast t$ such that $s^\ast sf=0$. Then by Proposition~\ref{magic}, $af$ is non-singular with magic element $tf$. But also $sf=0$, therefore  $|\supp(af)| < |\supp(a)|$, and so, by induction, $0\neq (af)\cap E(S)\subseteq (a)\cap E(S)$, which completes the proof.

Otherwise, we have that, for all $s \in \supp(a)$,  $s^\ast s = t^\ast t$ holds. By a dual argument, if $ss^\ast\neq tt^\ast$ for some $s\in \supp(a)$, then we have $(a)\cap E(S)\neq 0$.  So we may assume that both $ss^\ast=tt^\ast$ and $s^\ast s = t^\ast t$ hold for all $s\in \supp(a)$.

Suppose there exists some $t' \in \supp(a)$ which is not magic for $a$. Then there exists $0 \neq u \leq t'$ with $\sum\limits_{s \geq u}a_s=0$.
Then the coefficient of $t'u^\ast u=u$ in $au^\ast u$ is
\[\sum\limits_{su^\ast u=u}a_s=\sum\limits_{s \geq u}a_s=0,\]
so $|\supp(au^\ast u)|<|\supp(a)|$. Note that $tu^\ast u \neq 0$, as $t^\ast tu^\ast u=(t')^\ast (t')u^\ast u=u^\ast u\neq 0$, and so, by Proposition~\ref{magic}, $au^\ast u$ is non-singular with magic element $tu^\ast u$. Then by induction we obtain that $0\neq (au^\ast u)\cap E(S) \subseteq (a)\cap E(S)$, which completes the proof. For the rest of the proof we are only concerned with that case that all elements of $\supp(a)$ are magic, and $ss^\ast =tt^\ast$, $s^\ast s= t^\ast t$ for all $s,t\in \supp(a)$.

Choose two elements $s \neq t$ of $\supp(a)$. First assume that $ses^\ast=tet^\ast$ for all $e\in E(S)$.  Then by Lemma~\ref{l:qfundamental} we can find $0\neq u\leq s,t$ as $S$ is quasi-fundamental.  Putting $e=u^\ast u$, we have that $se=te=u\neq 0$ and so by Proposition~\ref{magic}, we have that $ae$ is non-singular.  But $|\supp(ae)|<|\supp(a)|$ as $se=te$.  Therefore, by induction, $0\neq (ae)\cap E(S)\subseteq (a)\cap E(S)$, as was required.

Otherwise, there exists $e\in E(S)$ with  $ses^\ast \neq tet^\ast$. Necessarily $se,te\neq 0$, for if $se=0$, then $t^\ast te=s^\ast se=0$, and so $ses^\ast=0=tt^\ast tet^\ast = tet^\ast$, and dually $te=0$ implies $ses^\ast=0=tet^\ast$. By Proposition~\ref{magic}, $se$ and $te$ are magic for $ae$. The coefficients of $0\neq se\leq s$ and $0\neq te\leq t$ in $ae$ are $\sum_{z \geq se}a_z$ and $\sum_{z \geq te}a_z$, respectively, neither of which are $0$ as $s,t$ are magic for $a$. Hence $se, te \in \supp(ae)$.

Without loss of generality, we may assume that $ses^\ast \not\geq tet^\ast$. Then $ses^\ast tet^\ast < tet^\ast$, and so, by the $0$-disjunctive property of $E(S)$, there must exist an idempotent $0 \neq f \leq tet^\ast$ such that $0=f tet^\ast ses^\ast=fses^\ast$. Note that $0 \neq f \leq tet^\ast$ implies $fte \neq 0$, whence $fae$ is non-singular with magic element $fte$ by Proposition~\ref{magic}. Furthermore, $fse=(fses^\ast)s=0$, and so $|\supp(fae)|<|\supp(a)|$.  Therefore, by induction, we have $0\neq (fae)\cap E(S) \subseteq (a)\cap E(S)$. This completes the proof.
\end{Proof}

A partial converse of Theorem~\ref{uniqueness} holds.

\begin{Lem}\label{l:converse.uniq}
Let $S$ be a non-zero inverse semigroup and $K$ a field.  Let  $I$ be the singular ideal of $\K S$ and assume that each ideal properly containing $I$ contains a non-zero idempotent $e\in E(S)$.  Then $S$ is quasi-fundamental.
\end{Lem}
\begin{proof}
Let $T=S/\mu$ where $\mu$ is the largest idempotent separating congruence on $S$ and let $J$ be the singular ideal of $\K T$.  Note that $T$ is non-zero as $S$ is non-zero and the projection $\gamma\colon S\to T$ is idempotent separating.  We claim that $\gamma(I)\subseteq J$, where we denote the extension $\K S\to \K T$ of $\gamma$ also by $\gamma$.  Indeed, suppose that $a\in I$ and let $0\neq e'\in E(T)$.  Then by Lallement's lemma, there is an idempotent $e\in S$ with $\gamma(e)=e'$ (take, for example, $e=s^\ast s$ where $\gamma(s)=e'$).  Then since $a$ is singular, there is $0\neq f\leq e$ with $af=0$.  As $\gamma$ is idempotent separating, $\gamma(f)\neq 0$ and so $0\neq \gamma(f)\leq e'$ satisfies $\gamma(a)\gamma(f)=\gamma(af)=0$.  Thus $\gamma(a)$ is singular. It follows that $\gamma^{-1}(J)$ is an ideal containing $I$. Note that $\gamma(E(S))\cap J=\{0\}$ by Proposition~\ref{embed:singular.quotient} and hence, since $\gamma$ is idempotent separating,  we have that $\gamma^{-1}(J)\cap E(S)=\{0\}$.  Therefore, $\gamma^{-1}(J)=I$ by hypothesis.     Suppose that $0\neq s,t\in S$ with $ses^\ast=tet^\ast$ for all $e\in E(S)$.  Then $\gamma(s)=\gamma(t)$ and so $s-t\in \gamma^{-1} (J)=I$.  Hence, we can find $0\neq e\leq s^\ast s$ with $(s-t)e=0$, i.e., $se=te$.  But $e=s^\ast se$ implies $se\neq 0$ and so $0\neq se\leq s,t$, whence $S$ is quasi-fundamental by Lemma~\ref{l:qfundamental}.
\end{proof}

The next theorem is our first main result.  It shows that the contracted algebra of a congruence-free inverse semigroup has a unique simple quotient and identifies that quotient.

\begin{Thm}
\label{simple}
Let $S$ be a congruence-free inverse semigroup with zero.  Then the singular ideal $I \lhd \K S$ is the unique maximal ideal of $\K S$ and hence $\K S /I$ is the unique simple quotient.
\end{Thm}
\begin{Proof}
By Theorem~\ref{uniqueness} if $J\lhd \K S$ is an ideal not contained in $I$, then there is an element $0\neq e\in J\cap E(S)$.  As $SeS=S$ by $0$-simplicity, we deduce that $J=\K S$.
\end{Proof}

\begin{Rem}
\label{r:more.general}
Theorem~\ref{simple} holds more generally if $S$ is quasi-fundamental, $0$-simple and has a $0$-disjunctive semilattice of idempotents with the exact same proof.
\end{Rem}

In~\cite{groupoidprimitive} an inverse semigroup was called tight if all its principal filters are tight (or, equivalently, all its filters are tight).  In concrete terms, this means that if $0\neq e$ and $e_1,\ldots, e_n<e$, then there exists $0\neq f\leq e$ with $fe_i=0$ for $i=1,\ldots, n$, i.e., $e$ has no non-trivial finite covers.    This is a stronger property than $E(S)$ being $0$-disjunctive and was considered by Munn quite early on in his study of simple contracted inverse semigroup algebras~\cite{Munnsimplealgebra}.  He called a semilattice with this property  \emph{strongly $0$-disjunctive} and so we shall use this term, as it is much older.  The tight ideal of $S$ is zero if and only if $E(S)$ is  strongly $0$-disjunctive by Corollary~\ref{c:tightpres}.  We are now prepared to characterize simplicity of contracted inverse semigroup algebras, that is, prove Theorem~B.

\begin{Cor}
\label{simplicity}
Let $S$ be an inverse semigroup with zero. Then $\K S$ is simple if and only if the following hold:
\begin{enumerate}
\item $S$ is fundamental;
\item $S$ is $0$-simple;
\item the singular ideal of $\K S$ is trivial.
\end{enumerate}
\end{Cor}

\begin{Proof}
If the singular ideal of $\K S$ is trivial, then $\tightid{K}{S}=0$ by Proposition~\ref{tight} and hence $S$ must have a strongly $0$-disjunctive semilattice of idempotents and so, in particular, a $0$-disjunctive semilattice of idempotents. But (1)--(2) and $0$-disjunctive semilattice are equivalent to $S$ being congruence-free by Theorem~\ref{simplesg}. So by Theorem~\ref{simple}, the conditions above do imply the simplicity of $\K S$. On the other hand, if $\K S$ is simple, than $S$ is clearly congruence-free (as proper quotients of $S$ naturally give rise to proper quotients of $\K S$).  Therefore, (1)--(2) hold by Theorem~\ref{simplesg}. Also the singular ideal, which is a proper ideal by Proposition~\ref{embed:singular.quotient}, must be trivial.  Thus (1)--(3) hold.
\end{Proof}

Note that we can replace fundamental by quasi-fundamental in Corollary~\ref{simplicity} by Remark~\ref{r:more.general}. In fact, quasi-fundamental and fundamental are equivalent as soon as the singular ideal vanishes, as the following proposition implies.

\begin{Prop}\label{p:qfundvsfund}
Let $S$ be a quasi-fundamental inverse semigroup and $0\neq s\in Z_S(E(S))$.  Then $s-s^\ast s$ is singular for any commutative ring $K$.
\end{Prop}
\begin{proof}
 Indeed, let $a=s-s^\ast s$ and  $e\in E(S)$ be a non-zero idempotent.  If $se=0$, then $ae=0$ and there is nothing to prove.  Otherwise, $0\neq se\in Z_S(E(S))$ and so $0\neq f\leq se\leq s$
 for some idempotent $f$ because $S$ is quasi-fundamental.  Then $f=fse$ and so $fe=f$, whence $f\leq e$.  Also $af = sf-s^\ast sf = f-s^\ast f=f-f=0$, as $f\leq s$ implies $f\leq s^\ast$.  Thus $a$ is singular.
\end{proof}

For a Hausdorff inverse semigroup, we can apply Proposition~\ref{tightissingular:hauss} to recover the following result of~\cite{groupoidprimitive}, which describes simplicity of $\K S$ in a semigroup theoretic way, independent of $K$, for Hausdorff inverse semigroups (including $0$-$E$-unitary inverse semigroups and all examples considered by Munn in~\cite{Munnsimplealgebra}).

\begin{Cor}
\label{simpleHausdorfsemigroup}
Let $S$ be a Hausdorff inverse semigroup with zero and $K$ any field.
\begin{enumerate}
\item  If $S$ is congruence-free, then the tight ideal is the unique maximal ideal of $\K S$.
\item $\K S$ is simple if and only if $S$ is $0$-simple, fundamental and $E(S)$ is strongly $0$-disjunctive.
\end{enumerate}
\end{Cor}

Note that Corollary~\ref{simpleHausdorfsemigroup} was proved in~\cite{groupoidprimitive} using groupoids, whereas the above proof is inverse semigroup theoretic.
We shall see later that, for non-Hausdorff inverse semigroups, simplicity of $\K S$ can depend on the characteristic of $K$ rather than on just algebraic properties of $S$, and so a characterization in terms of $S$ cannot be expected.
One might hope to find some structural description of inverse semigroups with a simple contracted inverse semigroup algebra in characteristic $0$ based on the semigroup structure and we pose that as a challenging problem for the future.

\section{Simplicity of ample groupoid algebras}\label{s:groupoidcase}

In~\cite{nonhausdorffsimple}, the authors give a complete characterization of simplicity of ample groupoid algebras in the case of a second countable ample groupoid. The Steinberg algebra of a second-countable ample groupoid $\G$ is simple if and only if $\G$ is minimal, effective and the ideal of singular functions, that is, of functions whose support has empty interior is trivial.  In fact, the proof in~\cite{nonhausdorffsimple} works for any ample groupoid that is minimal, effective and topologically principal.  In this section we extend their results to all ample groupoids (not just the topologically principal ones) via inverse semigroups.  We will need to work with Boolean inverse semigroups.   We shall use the uniqueness theorem (Theorem~\ref{uniqueness}) to show that the singular ideal of an additively congruence-free Boolean semigroup is the unique maximal ideal of the contracted semigroup algebra containing the tight ideal.  As a consequence, we obtain the results of~\cite{nonhausdorffsimple} for the general case in a topology-free way.  More generally, we characterize when the quotient of the algebra of a Boolean inverse semigroup by its singular ideal is simple.  This will be used to determine the ample groupoids whose algebras are simple modulo the ideal of singular functions.

\subsection{Boolean inverse semigroups}
We begin by showing that the construction $S\mapsto \tight S$ is functorial on the category of Boolean inverse semigroups and additive homomorphisms.  It will then follow that being additively congruence-free is a necessary condition for simplicity of $\tight S$.

\begin{Prop}
\label{functor}
The assignment $S\to \tight S$ is a functor from the category of Boolean inverse semigroups with additive homomorphisms to the category of $K$-algebras.  Moreover, the natural homomorphism $\eta \colon S\to \tight S$ is injective and is the universal additive homomorphism from $S$ to a $K$-algebra, where a homomorphism $\varphi$ is said to be additive if  it is zero-preserving and $ef=0$ implies $\varphi(e\vee f) = \varphi(e)+\varphi(f)$ for $e,f\in E(S)$.
\end{Prop}
\begin{Proof}
Recall that by non-commutative Stone duality~\cite{LawsonLenz}, if $\Gamma$ is the Boolean inverse semigroup of compact local bisections of $\G_T(S)$, then $s\mapsto (s,D(s^\ast s)\cap \widehat{E(S)}_T)$ is an isomorphism of $S$ with $\Gamma$.  Viewing $\tight S$ as the Steinberg algebra of $\G_T(S)$, we have that $\eta(s) = \chi_{(s,D(s^\ast s)\cap \widehat{E(S)}_T)}$ and hence $\eta$ is injective.  Corollary~\ref{c:grpdpres} shows the kernel $\tightid{K}{S}$ of the map $\K S\to \tight S$ induced by $\eta$ is generated by all $e+f -(e\vee f)$  with $ef=0$, where $e\vee f$ is taken in $E(S)$.  In particular, $\eta$ is the universal additive map of $S$ into a $K$-algebra.

To see that the assignment $S\to \tight S$ is a functor,  let $\varphi\colon S\to T$ be an additive homomorphism of Boolean inverse semigroups. Suppose that $ef=0$ with $e,f\in E(S)$.  Then $\varphi(e)\varphi(f)=0$ and $\varphi(e\vee f) = \varphi(e)\vee \varphi(f)$.  It follows that the homomorphism $\varphi\colon \K S\to \K T$ takes the generating set of $\tightid{K}{S}$ into $\tightid{K}{T}$ and hence $\varphi$ induces a homomorphism $\tight S\to \tight T$.
\end{Proof}

\begin{Rem}
\label{r:functor}
The proof shows that the effect of the functor on a morphism $\varphi\colon S\to T$ is the algebra homomorphism $\overline{\varphi}\colon \tight S\to \tight T$ given by $\overline{\varphi}(s+\tightid{K}{S}) = \varphi(s)+\tightid{K}{T}$.
\end{Rem}

Let us now restate the uniqueness theorem (Theorem~\ref{uniqueness}) in the context of Boolean inverse semigroups, taking into account Proposition~\ref{tight} (that the singular ideal contains the tight ideal), Proposition~\ref{tightissingular:hauss} and that Boolean algebras are $0$-disjunctive by Proposition~\ref{Booleanalg}.

\begin{Thm}[Uniqueness]
\label{uniqueness2}
Let $S$ be a quasi-fundamental Boolean inverse semigroup and $K$ a field. Let $I$ be the singular ideal of $\K S$ and $\tightid{K}{S}$ the tight ideal.  Identify $\tight S$ with $\K S/\tightid{K}{S}$.  We view $S$ as embedded in $\tight S$ in the natural way.  The following are equivalent for an ideal $J\lhd \tight S$.
\begin{itemize}
\item [(1)]  $J\cap S=\{0\}$.
\item [(2)]  $J\cap E(S)=\{0\}$.
\item [(3)]  $J\subseteq I/\tightid{K}{S}$.
\end{itemize}
Moreover, if $S$ is fundamental, then these are equivalent to the following:
\begin{itemize}
\item [(4)] The canonical mapping $S\to \tight S/J$ is injective;
\item [(5)]  The canonical mapping $E(S)\to \tight S/J$ is injective.
\end{itemize}
In particular, if $S$ is Hausdorff, then $J\cap E(S)\neq \{0\}$ for every non-zero ideal $J\lhd \tight S$.
\end{Thm}

It follows from Lemma~\ref{l:converse.uniq} that $S$ being quasi-fundamental is necessary for the conclusions of Theorem~\ref{uniqueness2} to hold.
In the additively congruence-free case, we obtain an analogue of Theorem~\ref{simple}.

\begin{Cor}
\label{uniquemax}
Let $S$ be a quasi-fundamental, additively $0$-simple Boolean inverse semigroup and $K$ a field.  Let $\tightid{K}{S}$ be the tight ideal of $\K S$ and $I$ the singular ideal.  Identify $\tight S$ with $\K S/\tightid{K}{S}$.  Then $I/\tightid{K}{S}$ is the unique maximal ideal of $\tight S$ and $(\K S/\tightid{K}{S})/(I/\tightid{K}{S})\cong \K S/I$ is the unique simple quotient of $\tight S$.  This applies in particular, if $S$ is additively congruence-free.
\end{Cor}
\begin{Proof}
Let $\eta\colon S\to \tight S$ be the natural embedding.  It is additive by Proposition~\ref{functor}.
If $J\lhd \tight S$ with $J\nsubseteq I/\tightid{K}{S}$, then $J$ contains an element of the form $\eta(e)$ with $0\neq e\in E(S)$ by Theorem~\ref{uniqueness2} since $S$ is quasi-fundamental.  Then $L=\eta^{-1}(J)$ is an additive ideal of $S$ containing $e$.   Since $S$ is additively $0$-simple by Theorem~\ref{addcongfree} and $L\neq 0$, we deduce that $L=S$ and hence $J= \tight S$ as $\eta(S)$ spans $\tight S$.  The final statement follows because an additively congruence-free Boolean inverse semigroup is fundamental and additively $0$-simple by Proposition~\ref{addcongfree}.
\end{Proof}

We now characterize simplicity of the essential algebra of a Boolean inverse semigroup.  This will later be translated into the language of ample groupoids.

\begin{Thm}
\label{simplicity3}
Let $S$ be a Boolean inverse semigroup and $K$ a field.  Let $I$ be the singular ideal of $\K S$.  Then $(\K S/\tightid{K}{S})/(I/\tightid{K}{S})\cong \K S/I$ is simple if and only if:
\begin{itemize}
\item [(1)] $S$ is quasi-fundamental;
\item [(2)] $S$ is additively $0$-simple.
\end{itemize}
\end{Thm}
\begin{Proof}
If $S$ is quasi-fundamental and additively $0$-simple, then $\K S/I$ is simple by Corollary~\ref{uniquemax}.  Suppose now that $\K S/I$ is simple.  Then the only ideal of $\K S$ properly containing $I$ is $\K S$, which clearly contains non-zero idempotents from $E(S)$.  Thus $S$ is quasi-fundamental by Lemma~\ref{l:converse.uniq}.

Next we show that $S$ is additively $0$-simple.  Let $J$ be a non-zero additive ideal of $S$ (resetting notation).  Note that $J\nsubseteq I$ by Proposition~\ref{embed:singular.quotient}.   Then $KJ$ is an ideal of $\K S$ not contained in $I$,  and so, by simplicity of $\K S/I$, we must have that $I+KJ=\K S$.  Let $e\in E(S)$.  Then $e=a+\sum_{t\in J} c_tt$ with $a\in I$ and $c_t\in K$.  Since $I$ and $J$ are ideals, multiplying both sides on the right by $e$ we may assume without loss of generality $t=te$ for all $t$ with $c_t\neq 0$.  Let $F=\{t\in J:c_t\neq 0\}$ and put  $f=\bigvee_{t\in F} t^\ast t$.  Then $f\in J$ because $J$ is an additive ideal and $f\leq e$ as $te=t$ for all $t\in F$ implies $t^\ast t=t^\ast te\leq e$.  Therefore, $e\setminus f = a(e\setminus f)+\sum_{t\in J} c_tt(e\setminus f) = a(e\setminus f)\in I$ as $t(e\setminus f)=0$ in $S$ for $t\in F$ because $t^\ast t\leq f$ by definition.  But then $e\setminus f=0$ by Proposition~\ref{embed:singular.quotient}.  Thus $e=f\in J$ and hence $E(S)\subseteq J$.  Therefore, $S=SE(S)=J$.  We conclude that $S$ is additively $0$-simple.
\end{Proof}

To describe simplicity of tight algebras of Boolean inverse semigroups, we observe that quasi-fundamental is equivalent to fundamental when the singular ideal coincides with the tight ideal.

\begin{Prop}\label{p:qfcoinf}
Let $S$ be a Boolean inverse semigroup such that $\tightid{K}{S}$ coincides with the singular ideal of $\K S$ for a commutative ring $K$.  Then $S$ is fundamental if and only if $S$ is quasi-fundamental.
\end{Prop}
\begin{proof}
It suffices to show that if $S$ is quasi-fundamental, then it is fundamental. Let $0\neq s\in Z_S(E(S))$. Then $s-s^\ast s$ is singular over $K$ by Proposition~\ref{p:qfundvsfund}.  Hence $s+\tightid{K}{S}=s^\ast s+\tightid{K}{S}$ and so $s=s^\ast s$ by Proposition~\ref{functor}.  It follows that $S$ is fundamental.
\end{proof}

The next result is the simplicity characterization of Steinberg algebras of ample groupoids in the language of Boolean inverse semigroups.

\begin{Cor}
\label{simplicity2}
Let $S$ be a Boolean inverse semigroup and $K$ a field.  Then $\tight S$ is simple if and only if:
\begin{itemize}
\item [(1)] $S$ is fundamental;
\item [(2)] $S$ is additively $0$-simple;
\item [(3)] the singular ideal of $\K S$ coincides with the tight ideal.
\end{itemize}
\end{Cor}
\begin{Proof}
Let $I$ denote the singular ideal of $\K S$ and $\tightid{K}{S}$ the tight ideal.  If (1)--(3) hold, then $\K S/\tightid{K}{S}=\K S/I$ is simple by Theorem~\ref{simplicity3}.  Conversely, assume that $\tight S$ is simple.  Since $\tightid{K}{S}\leq I\lneq \K S$, we must have $I=\tightid{K}{S}$ and hence $\K S/I$ is simple.  We deduce from Theorem~\ref{simplicity3} that $S$ is quasi-fundamental and additively $0$-simple.  But then $S$ is fundamental by Proposition~\ref{p:qfcoinf}.  We conclude that  (1)--(3) are necessary.
\end{Proof}

In particular, for Hausdorff Boolean inverse semigroups, having no non-trivial homomorphic images in the category of Boolean inverse semigroups is equivalent to having a simple tight algebra over some, or any, field.

\begin{Cor}
\label{Hausdorffsimplegroupoid}
If $S$ is a Hausdorff Boolean inverse semigroup and $K$ is a field, then $\tight S$ is simple if and only if $S$ is additively congruence-free.
\end{Cor}

\subsection{Ample groupoids}
We now translate the results of the previous subsection into the language of ample groupoids.  We then recover the main result of~\cite{nonhausdorffsimple} without the assumption of second countability or topological principality (and avoiding all topology), as well as the older results on uniqueness theorems and simplicity for Hausdorff groupoids~\cite{operatorsimple1,operatorguys2,groupoidprimitive}.

If $\G$ is an ample groupoid and $K$ is a commutative ring with unit, we define the \emph{support} of $f\colon \G\to K$ to be $\supp(f)=f^{-1}(K\setminus \{0\})$.  Note that we do not take a closure here as is typical in analysis.   A function $f\in K\G$ is defined to be \emph{singular} if the interior of $\supp(f)$ is empty.  If $\G$ is Hausdorff, then every non-zero function is non-singular, being locally constant.  The set of singular functions is an ideal that we denote by $\sing{K}{\G}$. The proof that $\sing{K}{\G}$ is an ideal  given in~\cite{nonhausdorffsimple} for fields works over any base ring.   The $C^*$-algebraic analogue of the  ideal of singular functions is studied in~\cite{EP19,KM}.

Let $S$ be an inverse semigroup and $K$ a commutative ring with unit.
Let $I \lhd \K S$ denote the singular ideal.
 Let $\Psi\colon \K S\to \tight S$ be the composition of the isomorphism of Theorem~\ref{isom} with the restriction homomorphism $\stein S\to \tight S$, where we are now viewing $\tight S$ as a Steinberg algebra.

\begin{Prop}
\label{identify:singular}
The singular ideal $I$ of $\K S$ is $\Psi^{-1}(\sing{K}{\G_T(S)})$.
\end{Prop}
\begin{Proof}
We show that $\Psi(I)=\sing{K}{\G_T(S)}$.  Since $\Psi$ is a surjective ring homomorphism, this will then imply that $\sing{K}{\G_T(S)}$ is an ideal (independently from~\cite{nonhausdorffsimple}). As the singular ideal contains the tight ideal by Proposition~\ref{tight}, we will then be able to conclude that  $I=\Psi^{-1}(\sing{K}{\G_T(S)})$ by the correspondence theorem.

Let $X=\widehat{E(S)}_T$ denote the tight spectrum of $E(S)$.
Let $a=\sum_{s\in \S}a_ss \in \K S$, and first suppose that $g:=\Psi(a)\notin \sing{K}{\G_T(S)}$; we show that $a\notin I$.  By assumption, $\supp(g)$ contains a non-empty basic open neighborhood $(t, D(e) \cap D(e_1)^c \cap \cdots \cap D(e_n)^c\cap X)$ with $e_1, \ldots, e_n <e\leq t^\ast t$. Necessarily, $e_1, \ldots, e_n$ do not cover $e$, by non-emptiness. So we obtain that there exists an idempotent $0 \neq f \leq e$ such that $fe_i=0$ for $i=1, \ldots, n$, and so  $(t, D(f)\cap X)\subseteq(t, D(e) \cap D(e_1)^c \cap \cdots \cap D(e_n)^c\cap X)\subseteq \supp(g)$.  Also $(t,D(f)\cap X)$ is non-empty because $f$ is contained in some ultrafilter. We claim that $af'\neq 0$ for all $0\neq f'\leq f$ and hence $a\notin I$  by (R). Indeed, $(t,D(f')\cap X)\subseteq (t,D(f)\cap X)\subseteq \supp(g)$. Let $\varphi$ be the character associated to an ultrafilter containing $f'$.  Then $[t,\varphi]\in (t,D(f')\cap X)$.  Now
\[\Psi(af')([t,\varphi]) = (g\ast \chi_{(f',D(f')\cap X)})([t,\varphi]) = g([t,\varphi])\neq 0.\]  Thus $af'\neq 0$.  So we have that $\psi(I)\subseteq \sing{K}{\G_T(S)}$.

Suppose now that $g\notin \Psi(I)$; we show that $g$ is non-singular.  Write again $g=\Psi(a)$ with
$a=\sum_{s\in \S}a_ss$, necessarily non-singular, say with magic element $t$ (note that $t\neq 0$). We claim that then $(t, D(t^\ast t)\cap X) \subseteq \supp (g)$. First note that $(t,D(t^\ast t)\cap X)\neq \emptyset$ because it contains any germ $[t,\chi_{\mathcal F}]$ where $\mathcal F$ is an ultrafilter containing $t^\ast t\neq 0$.  Let $[t, \varphi]\in(t, D(t^\ast t)\cap X)$.  In particular, $\varphi(t^\ast t)=1$. Consider the set
\[F=\{t\} \cup \{s: a_s \neq 0, [s, \varphi]=[t, \varphi]\}.\]
Since $F\setminus \{t\} \subseteq \supp(a)$, the set $F$ is finite.
For each $s \in F$, choose $u_s \leq s,t$ such that $\varphi(u_s^\ast u_s)=1$ (this exists since $[s,\varphi]=[t,\varphi]$).
Let $u=t \cdot \prod_{s \in F}u_s^\ast u_s\leq t$ and note that $u \leq s$ for all $s \in F$, and \[\varphi(u^\ast u)=\varphi \left(\prod\limits_{s \in F}t^\ast tu_s^\ast u_s\right)=1,\] whence $u\neq 0$ and
$[u, \varphi]=[t, \varphi]$.
Denote the set
$\{s \in \S: [s,\varphi]=[t, \varphi]\}$ by $Y$ and note that
$\supp(a) \cap Y \subseteq F \subseteq \{s \in \S: s\geq u\}\subseteq Y$.
Therefore, by Corollary~\ref{isomnice},
\[g([t,\varphi])=\sum\limits_{s \in Y}a_s=\sum\limits_{s \in Y\cap \supp(a)}a_s=
\sum\limits_{s\geq u}a_s,\]
but this is not $0$ as $t$ was magic for $a$ and $0\neq u \leq t$. So $[t, \varphi]\in\supp(g)$.    This completes the proof.
\end{Proof}

\begin{Rem}
\label{r:singular.ideal}
Since $\sing{K}{\G_T(S)}$ is an ideal by the previous proof, the singular functions form an ideal for any groupoid over any base commutative ring $K$, in light of the fact that every ample groupoid $\G$ is isomorphic to the tight groupoid of its inverse semigroup $\Gamma$ of compact local bisections.  Namely, the ideal of singular functions of  $K\G$ is the image of the singular ideal of $\K \Gamma$ under the natural surjective homomorphism $\K\Gamma\to K\G$ sending $U$ to $\chi_U$.
\end{Rem}

In analogy with the $C^*$-algebra setting~\cite{EP19,KM}, we define the \emph{essential algebra} of $\G$ over $K$ to be $K\G/\sing{K}{\G}$.

The following is the Cuntz-Krieger style uniqueness theorem for the essential algebras of not necessarily Hausdorff groupoids (without countability assumptions).

\begin{Thm}[Groupoid uniqueness]
Let $\G$ be a topologically free ample groupoid and $K$ a field.  Let $\sing{K}{\G}\lhd K\G$ be the ideal of singular functions.  Then any non-zero ideal of the essential algebra  $K\G/\sing{K}{\G}$ contains a coset $\chi_U+\sing{K}{\G}$ with $U$ a non-empty compact open subset of $\G^0$.  In particular, if $\G$ is Hausdorff, then any non-zero ideal of $K\G$ contains an element $\chi_U$ with $U$ a non-empty compact open subset of $\G^0$.
\end{Thm}
\begin{Proof}
In light of Proposition~\ref{translation} and Proposition~\ref{identify:singular}, this is immediate from Theorem~\ref{uniqueness2} applied to the Boolean inverse semigroup $\Gamma$ of compact local bisections, using the isomorphism $\G\cong \G_T(\Gamma)$.
\end{Proof}

Note that the condition that $\G$ is topologically free is necessary in order for the conclusion of the above theorem to hold by Lemma~\ref{l:converse.uniq} and Proposition~\ref{translation}.

The translation of Corollary~\ref{uniquemax} into the language of groupoids via Proposition~\ref{translation} is the following, again using the isomorphism $\G\cong \G_T(\Gamma)$.
\begin{Thm}
Let $\G$ be a topologically free and minimal ample groupoid and $K$ a field.  Let $\sing{K}{\G}\lhd K\G$ be the ideal of singular functions.  Then $\sing{K}{\G}$ is the unique maximal ideal of $K\G$ and $K\G/\sing{K}{\G}$ is simple.
\end{Thm}

Also, we have the following characterization of simplicity of the essential algebra of an ample groupoid, which is an exact analogue of the $C^*$-algebraic result in~\cite{KM} and was called Theorem~A$^\prime$ in the introduction.  The result is a direct translation of Theorem~\ref{simplicity3} into groupoid language via Proposition~\ref{identify:singular} and Proposition~\ref{translation}.

\begin{Thm}
\label{amplesimpleess}
Let $\G$ be an ample groupoid and $K$ a field.  Then the essential algebra $K\G/\sing{K}{\G}$ is simple if and only if $\G$ is minimal and topologically free.
\end{Thm}

\begin{Rem}
\label{r:cong.free}
Let $S$ be a congruence-free inverse semigroup. Then the essential algebra of $S$ is always simple by Theorem~\ref{simple} and hence the tight groupoid $\G_T(S)$ is always minimal and topologically free by Theorem~\ref{amplesimpleess}.
\end{Rem}

In fact, we show that if $S$ is quasi-fundamental and $E(S)$ is $0$-disjunctive, then the groupoid $\G_T(S)$ is always topologically free, generalizing~\cite[Corollary~5.11]{groupoidprimitive}  to the non-Hausdorff setting. This is in contrast to an example that we shall provide later showing that $S$ can be fundamental with $E(S)$ $0$-disjunctive, but the groupoid $\G_T(S)$ is not effective.

\begin{Cor}\label{c:top.free.from.qf}
Let $S$ be a quasi-fundamental inverse semigroups with $0$-disjunctive semilattice of idempotents.  Then the tight groupoid $\G_T(S)$ is topologically free.
\end{Cor}
\begin{proof}
Let $K$ be a field and $I$ the singular ideal of $\K S$.  Let $\Gamma$ be the inverse semigroup of compact local bisections of $\G_T(S)$.  Note since every non-zero idempotent is contained in an ultrafilter, if $0\neq e\in E(S)$, then the compact open set $D(e)\cap \widehat{E(S)}_T$ is non-empty.  We now apply Lemma~\ref{l:converse.uniq} to $\Gamma$ to show that $\Gamma$ is quasi-fundamental.

Via the isomorphism $\G_T(S)\cong \G_T(\Gamma)$ of~\cite{Exel},  the natural homomorphism $\lambda\colon \K S\to \K\Gamma$ given by $s\mapsto (s,D(s^*s)\cap \widehat{E(S)}_T)$ induces an isomorphism  $\K S/I\to \K\Gamma/\sing{K}{\G}$ by Proposition~\ref{tight} and Proposition~\ref{identify:singular}.  It follows that if $J$ is an ideal of $\K\Gamma$ properly containing $\sing{K}{\G}$, then $\lambda^{-1}(J)$ is an ideal of $\K S$ properly containing $I$.  Therefore, there is $0\neq e\in \lambda^{-1}(J)$ by Theoren~\ref{uniqueness}.  But then $0\neq D(e)\cap \widehat{E(S)}_T=\lambda(e)\in J$ and so we conclude that $\Gamma$ is quasi-fundamental by Lemma~\ref{l:converse.uniq}.  Therefore, $\G_T(S)$ is topologically free by Proposition~\ref{translation}.
\end{proof}

Note that if $K\G$ has no non-zero singular functions, then topological freeness is equivalent to effectiveness of $\G$ by Proposition~\ref{p:qfcoinf}, Proposition~\ref{identify:singular} and Proposition~\ref{translation}.
In particular, we recover the main result of~\cite{nonhausdorffsimple} without the second countability (or topological principality) assumption by translating Corollary~\ref{simplicity2} into the language of groupoids.  However, the proof of Corollary~\ref{simplicity2} is entirely algebraic.  This is one of the main results of the paper and was termed Theorem~A in the introduction.

\begin{Thm}
\label{amplesimple}
Let $\G$ be an ample groupoid and $K$ a field.  Then $K\G$ is simple if and only if:
\begin{itemize}
\item [(1)] $\G$ is effective;
\item [(2)] $\G$ is minimal;
\item [(3)] $K\G$ has no non-zero singular functions.
\end{itemize}
Moreover, effectiveness can be replaced by topological freeness in (1).
\end{Thm}

As a corollary we obtain the following classical result, first proved over the complex numbers in~\cite{operatorsimple1}, and in general in~\cite{operatorguys2,groupoidprimitive}.

\begin{Cor}
If $\G$ is a Hausdorff ample groupoid and $K$ is a field, then $K\G$ is simple if and only if $\G$ is effective and minimal.
\end{Cor}

\section{Descent and the singular ideal}

Fix an ample groupoid $\G$ for this section.  Our goal is to show that simplicity of $K\G$ depends only on the characteristic of $K$.  We do this by applying descent theory to the singular ideal.

In this section we shall use letters like $R$ to denote commutative rings with unit and reserve $K$ for fields.
If $\varphi\colon R_1\to R_2$ is a (unital) ring homomorphism, we have an induced base change homomorphism, which we denote abusively by $\varphi\colon R_1\G\to R_2\G$, given by $(\varphi(f))(\gamma) = \varphi(f(\gamma))$.  Note that this homomorphism is well defined because if $U$ is a compact local bisection, then $\varphi(\chi_U)=\chi_U$ and hence the spanning set of $R_1\G$ is mapped into the spanning set for $R_2\G$.  Note that $\varphi$ is injective (respectively, surjective) on the level of Steinberg algebras if and only if it is injective (respectively, surjective) on the level of rings.  The next proposition shows that the singular ideal is functorial in the base ring.

\begin{Prop}
\label{base:change}
If $\varphi\colon R_1\to R_2$ is a ring homomorphism, then $\varphi(\sing{R_1}{\G})\subseteq \sing{R_2}{\G}$.
\end{Prop}
\begin{Proof}
Suppose that $f\in R_1\G$ and $\varphi(f)$ is non-singular.  Then there is an open subset $U$ such that $\varphi(f)$ does not vanish on $U$. But this means that $\varphi(f(\gamma))\neq 0$ for all $\gamma\in U$ and hence $f$ does not vanish on $U$.  Therefore, $f$ is non-singular.
\end{Proof}

The next lemma is the trivial observation that the singular ideal restricts.

\begin{Lem}
\label{restriction}
Let $R_1$ be a subring of $R_2$.  Then $\sing{R_1}{\G} = R_1\G\cap \sing{R_2}{\G}$. In particular, if $\sing{R_2}{\G}=0$, then $\sing{R_1}{\G}=0$.
\end{Lem}
\begin{Proof}
It is clear from the definition of singularity that a mapping $f\colon \G\to R_1$ from $R_1\G$ is singular if and only if it is singular when viewed as a function to $R_2$.
\end{Proof}

It follows that non-triviality of the singular ideal passes from smaller rings to bigger rings. We now wish to show that non-triviality of the singular ideal descends from larger fields to smaller fields.  The key tool will be Galois descent.

Let $L/K$ be a Galois extension (perhaps infinite).  Then the Galois group $G=\Gal(L/K)$ is a profinite group with respect to the Krull topology.  An action of $G$ on a set $X$ is said to be \emph{continuous} if the stabilizer of each $x\in X$ is an open subgroup.  This is equivalent to the action map $G\times X\to X$ being jointly continuous, where $X$ is endowed with the discrete topology. For example, the natural action of $G$ on $L$ is continuous. See~\cite{Berhuy} for details.

A \emph{semilinear action} of $G$ on an $L$-vector space $V$ is a continuous action of $G$ by automorphisms of the additive group of $V$ such that $\sigma(\ell v) = \sigma(\ell) \sigma(v)$ for all $\ell\in L$ and $v\in V$.  Note that $G$ in fact acts by $K$-linear automorphisms of $V$.  The set \[V^G = \{v\in V: \sigma(v) = v, \forall \sigma\in G\}\] of $G$-fixed vectors is a $K$-vector subspace of $L$.

The canonical example of a semilinear action of $G$ is one of the form $V=L\otimes_K W$ where $W$ is a $K$-vector space and $\sigma(\ell\otimes w) = \sigma(\ell)\otimes w$.  Here $V^G= 1\otimes W\cong W$.  The Galois descent theorem says that all $G$-semilinear actions arise in this way~\cite[Theorem III.8.21]{Berhuy}.

\begin{Thm}[Galois descent]
Let $G=\Gal(L/K)$ have a continuous semilinear action on an $L$-vector space $V$.  Then the natural mapping $\Phi\colon L\otimes_K V^G\to V$ given by $\ell\otimes v\mapsto \ell v$ is a $G$-equivariant $L$-vector space isomorphism.  In particular, $V^G=0$ if and only if $V=0$.
\end{Thm}

We also recall the trace map.  If $E/K$ is a finite Galois extension with (finite) Galois group $H=\Gal(E/K)$, then $\mathrm{Tr}_{E/K}\colon E\to K$ is defined by $\mathrm{Tr}_{E/K}(a) = \sum_{\sigma\in H}\sigma(a)$ (this clearly takes values in $E^H=K$).  It is well known that the trace does not vanish identically on $E$.  Indeed, the elements of $H$ are linearly independent over $E$ by Dedekind's lemma on independence of characters and hence $\sum_{\sigma\in H}\sigma$ is not identically $0$.

We shall apply Galois descent to Steinberg algebras.  If $L/K$ is Galois, then each $\sigma\in G=\Gal(L/K)$ induces a ring automorphism $\sigma\colon L\G\to L\G$ via base change, and so $G$ acts on $L\G$.

\begin{Prop}
\label{ideal:descent}
If $L/K$ is Galois and $G=\Gal(L/K)$, then the action of $G$ on $L\G$ by base change automorphisms is continuous and semilinear.  Moreover, the singular ideal $\sing{L}{\G}$ is $G$-invariant.  Furthermore, $(L\G)^G = K\G$ and $(\sing{L}{\G})^G=\sing{K}{\G}$.
\end{Prop}
\begin{Proof}
The action is semilinear, for if $\ell\in L$, $f\in L\G$ and $\gamma\in G$, then \[\sigma(\ell f)(\gamma) = \sigma (\ell f(\gamma)) = \sigma(\ell) \sigma(f(\gamma)) = (\sigma(\ell)\sigma(f))(\gamma)\] and so $\sigma(\ell f)=\sigma(\ell) \sigma(f)$. To show continuity of the action, let $f\in K\G$.  Then $f=\sum_{i=1}^n c_i\chi_{U_i}$ with the $U_i$ compact local bisections and $c_i\in L$.  If $H_i$ is the stabilizer of $c_i$, then $H_i$ is open and the stabilizer of $f$ contains $H_1\cap\cdots \cap H_n$.  Therefore the stabilizer of $f$ is open (as any subgroup containing an open subgroup is open, being a union of cosets of that open subgroup).

Clearly, $K\G\subseteq (L\G)^G$.  Suppose $f\in (L\G)^G$ is non-zero.  Then $f(\gamma)\in K$ for all $\gamma\in \G$.  We must show that $f\in K\G$.  Write $f=\sum_{i=1}^n a_i\chi_{U_i}$ with the $U_i$ compact local bisections and $a_i\in L$.  Let $E$ be the Galois closure of $K(a_1,\ldots, a_n)$ in $L$  over $K$ (i.e., adjoin all roots of the minimal polynomials of the $a_i$ over $K$).  Then $E/K$ is a finite Galois extension and $f\in E\G$. The finite Galois group $H=\Gal(E/K)$ acts semilinearly on $E\G$ and fixes $f$, as $f$ takes values in $K$.  Choose $a\in E$ with $0\neq \mathrm{Tr}_{E/K}(a)\in K$.  Then
\[\mathrm{Tr}_{E/K}(a)f = \sum_{\sigma\in H}\sigma(a)f =   \sum_{\sigma\in H}\sigma(af) = \sum_{i=1}^n \sum_{\sigma\in H}\sigma(aa_i)\chi_{U_i}=\sum_{i=1}^n\mathrm{Tr}_{E/K}(aa_i)\chi_{U_i}\in K\G\]  and so $f\in K\G$.  Thus $K\G=(L\G)^G$.

Proposition~\ref{base:change} shows that $\sing{L}{\G}$ is $G$-invariant.  Hence $\sing{L}{\G}^G = \sing{L}{\G}\cap (L\G)^G = \sing{L}{\G}\cap K\G = \sing{K}{\G}$ by Lemma~\ref{restriction}.
\end{Proof}

The following is an immediate consequence of Galois descent and Proposition~\ref{ideal:descent}.

\begin{Cor}
\label{descent1}
If $L/K$ is Galois, then $L\G\cong L\otimes_K K\G$ and $\sing{L}{\G}\cong L\otimes_K \sing{K}{\G}$.
\end{Cor}

An application of Corollary~\ref{descent1} is the following lemma.

\begin{Lem}
\label{galois:case}
Let $L/K$ be a Galois extension.  Then $\sing{K}{\G}=0$ if and only if $\sing{L}{\G}=0$.
\end{Lem}

Let $K$ be a field of characteristic $p>0$.  Then a field extension $L/K$ is purely inseparable if, for all $a\in L$, there exists $n\geq 0$ with $a^{p^n}\in K$.  A field extension $L/K$ in characteristic zero is considered purely inseparable if and only if $L=K$.

\begin{Lem}
\label{purelyinsep}
Suppose that $K$ is a field of characteristic $p>0$ and that $L/K$ is purely inseparable.  If $\sing{K}{\G}=0$, then $\sing{L}{\G}=0$.
\end{Lem}
\begin{Proof}
Let $\Phi\colon L\to L$ be the Frobenius endomorphism $\Phi(a)=a^p$.  Since $L$ is a field, $\Phi$ is injective.  We then have an induced injective homomorphism $\Phi\colon L\G\to L\G$ by base change. Let $f\in \sing{L}{\G}$ and write $f=\sum_{i=1}^n c_i\chi_{U_i}$ with the $U_i$ compact local bisections and the $c_i\in L$.  Then because $L/K$ is purely inseparable, there exists $r\geq 0$ with $\Phi^r(c_i)\in K$ for all $i=1,\ldots, n$.  Then $\Phi^r(f) = \sum_{i=1}^n\Phi^r(c_i)\chi_{U_i}\in \sing{L}{\G}\cap K\G = \sing{K}{\G}=0$ by Proposition~\ref{base:change} and  Lemma~\ref{restriction}.  Since $\Phi^r$ is injective, we conclude that $f=0$, as was required.
\end{Proof}

Recall that an extension $L/K$ is purely transcendental if $L\cong K(X)$ as a $K$-algebra where $X$ is a set of variables.

\begin{Lem}
\label{purelytrans}
If $L/K$ is purely transcendental and $\sing{K}{\G}=0$, then $\sing{L}{\G}=0$.
\end{Lem}
\begin{Proof}
Without loss of generality, we may assume that $L=K(X)$ with $X$ a set of variables.  First we consider the case that $K$ is infinite.  If $f\in \sing{K(X)}{\G}$, then \[f=\sum_{i=1}^m \frac{p_i(x_1,\ldots, x_n)}{q_i(x_1,\ldots, x_n)}\chi_{U_i}\] with the $U_i$ compact local bisections,  $x_1,\ldots, x_n\in X$ and $p_i,q_i\in K[x_1,\ldots, x_n]$ with the $q_i\neq 0$.  Putting $q=q_1\cdots q_m$, we have that $g=qf\in K[x_1,\ldots, x_n]\G\cap \sing{K(X)}{\G}=\sing{K[x_1,\ldots, x_m]}{\G}$ (by Lemma~\ref{restriction}).  For each $\vec a=(a_1,\ldots, a_n)\in K^n$, we have a homomorphism $\varepsilon_{\vec a}\colon K[x_1,\ldots,x_n]\to  K$ given by $\varepsilon_{\vec a}(p) = p(a_1,\ldots, a_n)$.  Moreover, because $K$ is infinite we have that $p=0$ if and only if $\varepsilon_{\vec a}(p)=0$ for all $\vec a\in K^n$.  By base change, we have a corresponding homomorphism $\varepsilon_{\vec a}\colon K[x_1,\ldots x_n]\G\to K\G$ and, by Proposition~\ref{base:change}, we have that $\varepsilon_{\vec a}(g)\in \sing{K}{\G}=0$ for all $\vec a\in K^n$.  We deduce that $g=0$, and hence $f=0$.  Thus $\sing{K(X)}{\G}=0$.

Next assume that $K$ is finite.  Let $\overline K$ be an algebraic closure of $K$.  Then $\overline K$ is infinite and is Galois over $K$, as $K$ is perfect.  Therefore, $\sing{\overline K}{\G}=0$ by Lemma~\ref{galois:case}.  Thus, $\sing{\overline K(X)}{\G}=0$ by the previous case and hence $\sing{K(X)}{\G}=0$ by Lemma~\ref{restriction}.  This completes the proof.
\end{Proof}

We are now ready to prove that vanishing of the singular ideal depends only on the characteristic of the field.

\begin{Thm}
\label{fulldescent}
Let $L/K$ be a field extension and $\G$ an ample groupoid.  Then $\sing{L}{\G}=0$ if and only if $\sing{K}{\G}=0$.  In particular, vanishing of the singular ideal depends only on the characteristic.
\end{Thm}
\begin{Proof}
If $\sing{L}{\G}=0$, then $\sing{K}{\G}=0$ by Lemma~\ref{restriction}.  Suppose that $\sing{K}{\G}=0$. Fix an algebraic closure $\overline L$ of $L$. Choose a transcendence basis $S$ of $\overline L/K$.  Then $K(S)/K$ is purely transcendental and $\overline L/K(S)$ is algebraic.  By Lemma~\ref{purelytrans}, we have that $\sing{K(S)}{\G}=0$.  Let $K(S)^{\mathrm{sep}}$ be the separable closure of $K(S)$ in $\overline L$.  Then $K(S)^{\mathrm{sep}}/K(S)$ is Galois and $\overline L/K(S)^{\mathrm{sep}}$ is purely inseparable ($\overline L=K(S)^{\mathrm{sep}}$ if the characteristic is $0$).  By Lemma~\ref{galois:case}, we have that $\sing{K(S)^{\mathrm{sep}}}{\G}=0$ and hence $\sing{\overline L}{\G}=0$ by Lemma~\ref{purelyinsep}. Therefore, $\sing{L}{\G}=0$ by another application of Lemma~\ref{restriction}. This completes the proof.
\end{Proof}

\begin{Cor}
\label{reducetoprime}
If $\G$ is an ample groupoid and $L$ is a field with prime field $K$, then $L\G$ is simple if and only if $K\G$ is simple. In particular, simplicity depends only on the characteristic of the field.
\end{Cor}
\begin{Proof}
By Theorem~\ref{amplesimple}, in order for the Steinberg algebra of $\G$ over a field to be simple, we must have that $\G$ is minimal and effective, and in that case, simplicity boils down to the singular ideal being trivial by Theorem~\ref{amplesimple}.  The corollary then follows from Theorem~\ref{fulldescent}.
\end{Proof}

\begin{Cor}
If $S$ is an inverse semigroup with zero and $L$ is a field with prime field $K$, then $L_0S$ is simple if and only if $\K S$ is simple. In particular, simplicity depends only on the characteristic of the field.
\end{Cor}
\begin{Proof}
This follows from Corollary~\ref{reducetoprime} because the contracted semigroup algebra of an inverse semigroup is the Steinberg algebra of its (contracted) universal groupoid.  Alternatively, the above descent method could be applied directly to inverse semigroups and the proof would be easier since $\S$ is a basis for the contracted inverse semigroup algebra over any field, and so $L_0S=L\otimes_K \K S$ for any field extension $L/K$.
\end{Proof}

To prove that the existence of non-zero singular functions descends from characteristic $0$ to characteristic $p$, we need to do some preliminary work.

Let $R$ be a commutative ring with unit and $I$ an ideal. We write $\pi\colon R\to R/I$ for the quotient map. Let $\G$ be an ample groupoid with inverse semigroup $\Gamma$ of compact local bisections.  Denote by $I\G$ the set of all elements of the form $\sum_{U\in \Gamma} c_U\chi_U$ with the $c_U\in I$.  Denote by $I_0\Gamma$ the ideal in $R_0\Gamma$ whose elements are linear combinations $\sum_{U\in \Gamma^\sharp}c_UU$ with $c_U\in I$.  Notice that $I_0\Gamma$ is the kernel of the natural reduction homomorphism $\pi\colon R_0\Gamma\to [R/I]_0\Gamma$.

\begin{Lem}\label{l:reduct}
Let $\G$ be an ample groupoid and $I$ an ideal of $R$.  Then $I\G$ is the kernel of the reduction homomorphism $\pi\colon R\G\to [R/I]\G$ induced by the canonical surjection $\pi\colon R\to R/I$.
\end{Lem}
\begin{proof}
If $f\in I\G$, then $f=\sum_{U\in \Gamma}a_U\chi_U$ with $a_U\in I$ and $\pi(f(\gamma))= \sum_{\gamma\in U} \pi(a_U)=0$.  Suppose conversely that $f\in \ker \pi$, that is, $f(\gamma)\in I$ for all $\gamma\in G$.  Write $f=\sum_{U\in \Gamma}a_U \chi_U$ and let $a=\sum_{U\in \Gamma^\sharp}a_UU\in R_0\Gamma$.

For a ring $R'$, let $\Psi_{R'}\colon R'_0\Gamma\to R'\G$ be the canonical surjection given by $\Psi_{R'}(U)=\chi_U$,  and note that $\ker \Psi_{R'}$ is generated as an ideal by the elements of the form $U+V-(U\cup V)$ with $U,V\subseteq \G^0$ disjoint compact open sets by Corollary~\ref{c:grpdpres}.  It follows that under the  natural surjection $\pi\colon R_0\Gamma\to [R/I]_0\Gamma$, we have that $\pi(\ker \Psi_R) = \ker \Psi_{R/I}$ as $\pi(U+V-(U\cup V)) = U+V - (U\cup V)$.  Also observe that $\pi\Psi_R = \Psi_{R/I}\pi$ by definition.

 Since $\Psi_R(a)=f$, we have that $\Psi_{R/I}(\pi(a)) = \pi(f) =0$ and so $\pi(a)\in \ker \Psi_{R/I}$.   Thus we can find $b\in \ker \Psi_R$ with $\pi(b)=\pi(a)$, that is, with $a-b\in \ker \pi=I_0\Gamma$.  Thus $a=b+c$ where $c=\sum_{U\in \Gamma^\sharp}c_UU$ with $c_U\in I$.  As $b\in \ker \Psi_R$ we deduce that $f=\Psi_R(a)=\Psi_R(c)=\sum_{U\in \Gamma^\sharp}c_U\chi_U\in I\G$.  The result follows.
\end{proof}

For a prime $p$, let $\mathbb F_p$ denote the field of $p$ elements.

\begin{Cor}
\label{c:p-adic}
Let $f\in \mathbb Z\G$ and let $\pi\colon \mathbb Z\G\to \mathbb F_p\G$ be the homomorphism induced by reduction modulo $p$ with $p$ prime.  If $f\neq 0$, then $f=p^mh$ with $h\in \mathbb Z\G$ and $\pi(h)\neq 0$, for some $m\geq 0$.
\end{Cor}
\begin{proof}
Note that since each characteristic function of a compact local bisection takes on only finitely many values, $f(\G)$ is finite.  Note that if $f\in p^n\mathbb Z\G$, then $p^n$ divides each element of $f(\G)$.  Since $f\neq 0$, it follows that there is a largest $m\geq 0$ with $f\in p^m\mathbb ZG$.   So $f = p^m h$ with $h\in \mathbb ZG$ and $h\notin p\mathbb ZG$.  Thus $\pi(h)\neq 0$ by Lemma~\ref{l:reduct}.
\end{proof}

\begin{Cor}
\label{c:any.field.then.c}
Let $\G$ be an ample groupoid.  If $\sing{F}{\G}=0$ for some field of positive characteristic, then $\sing{K}{\G}=0$ for all fields $K$ of characteristic $0$.  Hence $\G$ has a simple algebra over fields of characteristic $0$ whenever it has a simple algebra over some field of positive characteristic.
\end{Cor}
\begin{Proof}
For the first statement, by Theorem~\ref{fulldescent}, it suffices to show that if $\sing{\mathbb Q}{\G}\neq 0$, then $\sing{\mathbb F_p}{\G}\neq 0$ for all primes $p$.  Suppose $0\neq f\in \sing{\mathbb Q}{\G}$.  Write \[f= \sum_{i=1}^n \frac{r_i}{s_i}\chi_{U_i}\] with the $U_i$ compact local bisections and the $r_i,s_i\in \mathbb Z\setminus \{0\}$.  Then, putting $s=s_1\cdots s_n$, we have that $0\neq sf\in \mathbb Z\G$. Then by Corollary~\ref{c:p-adic} we can write $sf=p^mh$ with $h\in \mathbb Z\G$ and $\pi(h)\neq 0$ where $\pi\colon \mathbb Z\G\to \mathbb F_p\G$ is the reduction homomorphism.  Note that $h = \frac{s}{p^m}f\in \mathbb Z\G\cap \sing{\mathbb Q}{\G}=\sing{\mathbb Z}{\G}$ (by Lemma~\ref{restriction}) and so $0\neq \pi(h)\in \sing{\mathbb F_p}{\G}$ by Proposition~\ref{base:change}.   Thus $\sing{\mathbb F_p}{\G}\neq 0$, as required.  The final statement follows from Corollary~\ref{reducetoprime} and Theorem~\ref{amplesimple}.
\end{Proof}

\section{Examples from self-similar group actions}

We close the paper with examples that illustrate some unexpected phenomena regarding the simplicity of non-Hausdorff ample groupoid and inverse semigroup algebras.
For Hausdorff inverse semigroups and ample groupoids, simplicity of their algebras is independent of the field: it depends only on properties of the semigroup or groupoid.
Clark \textit{et al.}~\cite{nonhausdorffsimple} give an example of a non-Hausdorff ample groupoid whose Steinberg algebra is non-simple over fields of characteristic 2 but simple otherwise coming from the Nekrashevych algebra of the Grigorchuk group, which is a self-similar group (see also~\cite{Nekrashevychgpd}). However, this example is not of the form $\stein S$ with $\G(S)$ the universal groupoid of an inverse semigroup $S$, so whether simplicity of inverse semigroup algebras could depend on the characteristic of the field was still unknown.

We provide here the first examples of inverse semigroups whose contracted algebras are simple over fields of some characteristics, but not over others. Our examples show that Corollary~\ref{c:any.field.then.c} contains the only implication between simplicity over fields of different characteristics.  We provide both second countable examples and also examples that are minimal and effective but not topologically principal, and therefore cannot be handled by the methods of~\cite{nonhausdorffsimple,Nekrashevychgpd}. We also construct the first examples of congruence-free inverse semigroups with a strongly $0$-disjunctive semilattice of idempotents whose contracted semigroup algebras fail to be simple over any field. Our first such example has a universal (equivalently, tight) groupoid which is not effective, which disproves a claim in~\cite{LalondeMilan}. We also give a family of examples where the universal (equivalently, tight) groupoid is effective, thereby answering a question posed in~\cite{nonhausdorffsimple}, which had been independently answered by Nekrashevych in an unpublished result. These examples arise from self-similar group actions over infinite alphabets and are, in fact, inverse hulls of left cancellative monoids with least common multiples.
We also give the first examples of  minimal and effective ample groupoids that are not topologically principal, both in the Hausdorff and non-Hausdorff settings.  These examples are tight groupoids associated to self-similar group actions.

   The theory of self-similar groups had its origins in the work of Glushkov, Aleshin, Grigorchuk, Sushchanskii, Sidki, and others, on torsion groups generated by finite automata (see~\cite{GNS} for some history)  before being formalized by Nekrashevych~\cite{selfsimilar}.  However, self-similar groups can also be traced all the way back to thesis of Perrot in 1972 (see~\cite{ExPadKatsura,LawsonCorrespond} for details), and in the context of semigroup algebras, examples of this form appear in Munn's work on simple contracted inverse semigroup algebras~\cite{Munntwoexamples,Munnsimplealgebra}, both before self-similar groups were introduced in their modern form.

\subsection{Self-similar group actions}
Recall that a \emph{self-similar group action} of a group $G$ over an alphabet $X$ (of cardinality at least $2$) is defined as follows~\cite{selfsimilar}.   One has a mapping $G\times X\to X$, denoted $(g,x)\mapsto g(x)$, and a mapping $G\times X\to G$, denoted $(g,x)\mapsto g|_x$.  One can then extend the two mappings to words in the free monoid $X^*$  on $X$ via $g(xu) = g(x)g|_x(u)$ and $g|_{xu} = (g|_x)|_u$ for $x\in X$ and $u\in X^*$.  Note that $g(\varepsilon)=\varepsilon$ and $g|_{\varepsilon}=g$, where $\varepsilon$ denotes the empty word.  One observes that $G\times X^*\to G$ is a right action of the free monoid $X^\ast$ on $G$ (as a set) and one requires that $G\times X^*\to X^*$ is a left action of $G$ on $X^*$ (as a set),  $1|_x=1$ for all $x\in X$ and $(gh)|_x = g|_{h(x)}h|_x$ for all $g,h\in G$ and $x\in X$.  See~\cite{LawsonCorrespond} for details.  We say that the self-similar action of $G$ is \emph{faithful} if $G$ acts faithfully on $X^*$, in which case one says that $G$ is a self-similar group. A faithful self-similar group action of $G$ on $X^\ast$ is the same thing as a faithful, length-preserving action such that, for every $g\in G$ and $x\in X$, there is an element $g|_x\in G$ with $g(xw)=g(x)g|_x(w)$ for all $w\in X^\ast$.

The self-similar action also defines an action of $G$ on the Cayley graph of $X^\ast$ (the $|X|$-regular rooted tree) by root fixing automorphisms, and it is often convenient to have that picture in mind.
Notice that this action naturally extends to $X^\omega$ (the boundary of the tree) via the formula
\begin{equation}\label{eq:boundary.form}
g(x_1x_2x_3 \ldots )=g(x_1)g|_{x_1}(x_2)g|_{x_1x_2}(x_3)\ldots.
\end{equation}

Let $M=X^*G$.  This a left cancellative monoid in which each pair of elements with a common multiple has a least common multiple, and which can also be understood as the Zappa-Sz\'ep product of $X^\ast$ and $G$~\cite{LawsonWallis}.   The product in $M$ is given by $(ug)(vh) = (ug(v))(g|_v h)$, where $g,h\in G$ and $u,v\in X^*$. Since $M$ is left cancellative, its left regular action is by injective mappings and hence $M$ embeds in the symmetric inverse monoid on $M$.  The inverse monoid $S$ generated by $M$ is called the \emph{inverse hull} of $M$.  Its non-zero elements are uniquely of the form $ugv^\ast$ with $u,v\in X^*$ (where $v^\ast$ is the inverse map to left multiplication by $v$) and $g\in G$ because of the least common multiple property; see~\cite{ExelSteinbergHull} for more details on inverse hulls.

 Formally, $S$ consists of a zero element $0$ and  products of the form
$ugv^\ast$ where $u, v \in X^\ast$,  $g \in G$, and
multiplication is defined by the rules
\[
\begin{array}{ccc}
gx=g(x)g|_{x}, &
x^\ast g=g|_{g^{-1}(x)}(g^{-1}(x))^\ast, &
x^\ast y=\begin{cases}
1 & \hbox{if}\ x=y \\
0 & \hbox{otherwise}
\end{cases}
\end{array}
\]
for $x,y\in X$.
We can immediately see that $u g v^\ast \cdot w h z^\ast \neq 0$ if and only if $v$ and $w$ are prefix comparable, in which case one can use the rules to reduce the product to the normal form.
The submonoid of $S$ consisting of $0$ and all elements of the form $uv^\ast$ with $u,v\in X$ is well known and is called the polycyclic inverse monoid $P_X,$  cf.~\cite[Chapter~9]{Lawson}.

Notice that the non-zero idempotents of $S$ are the elements of the form $ww^\ast$ with $w\in X^*$ and that $ww^\ast\leq uu^\ast$ if and only if $u$ is a prefix of $w$.  It was shown by Munn~\cite{Munnsimplealgebra} that $E(P_X)$, which coincides with $E(S)$, is strongly $0$-disjunctive if and only if $X$ is infinite.  The filters on $E(P_X)$ are well known to be in bijection with the set $X^\ast\cup X^\omega$ of both finite and (right) infinite words over the alphabet $X$.  The filter associated to a word $w$ consists of all $uu^\ast$ with $u$ a prefix of $w$.  The ultrafilters correspond to the infinite words and the principal filters (those generated by a single idempotent) to the finite words.  If $X$ is infinite, all these filters are tight, whereas if $X$ is finite only the ultrafilters are tight.

It is shown in~\cite[Proposition~6.2]{LawsonCorrespond} that if $|X|\geq 2$ and the action is faithful, then $S$ is congruence-free, but this is also not hard to see directly.
There is an an action of $S$ on $X^\ast$ by partial one-to one maps coming from the actions of $X^\ast$ and $G$ on $X^*$:
\begin{nalign}
\label{spectaction}
u gv^\ast \colon vX^\ast &\to uX^\ast\\
vw &\mapsto u g(w).
\end{nalign}
It is faithful if and only if the self-similar action of $G$ is faithful.
Moreover, under the identification of $E(S)$ and $X^\ast$, this is precisely the classical Munn representation of $S$~\cite{Lawson}.  Since an inverse semigroup is fundamental if and only if its Munn representation is faithful, this shows that $G$ acts faithfully on $X^*$ if and only if $S$ is fundamental.    Since $u^\ast u=1$, it is obvious that all the non-zero idempotents generate the same ideal (in fact, are in the same Green's $\mathcal D$-class) and so $S$ is $0$-simple.  It is also easily verified to be $0$-disjunctive as long as $|X|\geq 2$ since if $u$ is a proper prefix of $v$, say $v=uz$, and $x\in X$ differs from the first letter of $z$, then $(ux)(ux)^\ast vv^\ast=0$. Thus faithful self-similar actions give rise to congruence-free inverse semigroups, where the simplicity of the contracted semigroup algebra can only fail by having a non-trivial singular ideal.  Moreover, if the alphabet is infinite, the semilattice of idempotents is strongly $0$-disjunctive and so the tight and universal groupoids of our inverse semigroup coincide.

We also describe the action of $S$ on its spectrum $\widehat{E(S)}$.  If we identify $\widehat{E(S)}$ with $X^\ast\cup X^{\omega}$, then
\begin{nalign}
\label{spectralaction}
u gv^\ast \colon v(X^\ast\cup X^{\omega}) &\to u(X^\ast\cup X^{\omega})\\
vw &\mapsto u g(w)
\end{nalign}
where $g$ acts as per \eqref{eq:boundary.form} in the case $w\in X^\omega$ (since the prefixes of $g(w)$ are precisely the images of the prefixes of $w$ under $g$).  Under this identification the basic compact open set $D(ww^\ast)$ of $\widehat{E(S)}$ corresponds to $w(X^\ast\cup X^{\omega})$, and so we shall denote this latter set by $D(w)$ for succinctness. A basic compact  open neighborhood is a set of the form \[D(u)\cap D(uw_1)^c \cap \ldots \cap D(uw_n)^c,\]
which contains all finite and infinite words that start with $u$ not followed by $w_i$, $1 \leq i \leq n$, (with possibly $n=0$).

Let us recall the structure of the universal groupoid $\G(S)$ of an inverse semigroup $S$ coming from a self-similar group action.
The universal groupoid $\G(S)$ has $X^\ast \cup X^\omega$ as its set of objects, the arrows are germs are of the form $[\alpha g\beta^\ast, \beta w]$ with $d([\alpha g\beta^\ast, \beta w])=\beta w$, $r([\alpha g\beta^\ast, \beta w])=\alpha g(w)$. A basic compact open neighborhood of a germ $[\alpha g\beta^\ast, \beta w]$ is a compact local bisection $(\alpha g \beta^\ast, U)$, where $U$ is a basic compact open set in $X^\ast\cup X^\omega$, as described above, containing $\beta w$.
If $|X|<\infty$, then the tight groupoid of $S$ is the restriction to the closed invariant subspace $X^{\omega}$. If $|X|=\infty$, then the tight groupoid and the universal groupoid coincide.
From the point of view of simplicity of Steinberg algebras, we are therefore interested in tight groupoids associated to actions over finite alphabets and universal (equivalently, tight) groupoids associated to actions over infinite alphabets. The latter are of special interest to us as their algebras are isomorphic to contracted inverse semigroup algebras. We focus on the case of a finite alphabet in~\cite{simplicityNekr}, where we give a more direct characterization of the simplicity of the algebra of the tight groupoid and provide an algorithm to decide simplicity of the algebra for contracting self-similar groups.  We also give a concrete representation of the essential algebra of $S$ over $K$ as an algebra of operators on $KX^{\omega}$.

For any $X$, the tight groupoid of $S$ is minimal and topologically free by Remark~\ref{r:cong.free}, as $S$ is congruence-free. If $|X|$ is finite, it is always effective by~\cite[Section~17]{ExPadKatsura}\footnote{The paper~\cite{ExPadKatsura} uses the term ``essentially principal'' for effective.}.  We proceed to characterize effectiveness of the universal groupoid when $X$ is infinite.

We say that a word $w\in X^\ast$ is \emph{strongly fixed} by $g\in G$ if $g(w)=w$ and $g|_w=1$; if the action is faithful this is equivalent to saying that $g$ fixes $wX^\ast$ (or equivalently $w(X^\ast\cup X^\omega)$). Notice that $gww^\ast=g(w)g|_ww^\ast$ is idempotent if and only if $g$ strongly fixes $w$, explaining the importance of this notion.

\begin{Prop}
\label{prop:effective}
Let $G$ be a group with a faithful self-similar action on $X^\ast$ where $X$ is infinite, and let $S$ be the corresponding inverse semigroup. Then $\G(S)$ is topologically free. Furthermore $\G(S)$ is effective if and only if whenever the set of letters strongly fixed by $g\in G$ is cofinite in $X$, we must have that $g$ is the identity.
\end{Prop}
\begin{proof}
We first begin by describing the basic open sets in $\Is(\G(S))$. Assume
\[U=(\alpha g \beta ^\ast, D(\beta w)\cap D(\beta w w_1)^c \cap \ldots \cap D(\beta ww_n)^c)\]
is any basic compact local bisection contained in $\Is(\G(S))$. Then, if $u$ is a finite word, we have $\alpha g\beta^\ast(\beta u)=\beta u$ if and only if $\alpha g(u)=\beta u$.  Since $|g(u)|=|u|$, this occurs if and only if $\alpha=\beta$ and $g(u)=u$.  In particular, since $[\alpha g\beta^\ast,\beta w]\in U$ is an isotropy element, we deduce that $\alpha=\beta$, $g(w)=w$ and $U$ is of the form
\[(\beta g \beta ^\ast, D(\beta w)\cap D(\beta w w_1)^c \cap \ldots \cap D(\beta ww_n)^c)\]
where for any $u\in D(w)$, of which none of $w_1,\ldots, w_n$ is a prefix, we have $g(u)=u$.

Although we have already observed that $\G(S)$ is topologically free as a consequence of Remark~\ref{r:cong.free}, we provide here a short direct proof.
 Denote the first letter of $w_i$ by $x_i$, and suppose that $x \in X\setminus\{x_1, \ldots, x_n\}$.
Then $[\beta g \beta^\ast, \beta wx]=[\beta g \beta^\ast \beta wx(\beta wx)^\ast, \beta wx]=[\beta wx  (\beta wx)^\ast, \beta wx] \in U \cap \G(S)^0$, so $U \cap \G(S)^0\neq \emptyset$, which shows that $\int(\Is(\G(S))\setminus \G(S)^0)=\emptyset$ that is $\G(S)$ is topologically free.

Assume furthermore that whenever the set of letters strongly fixed by a group element is cofinite in $X$, than that element is the identity.
Note we also have $[\beta g\beta^\ast,\beta wxv]\in U$ for all $v\in X^\ast$, and so
$wxv=g(wxv)=wg|_w(xv)$ for any $v \in X^\ast$, that is, $g|_w$ strongly fixes $X\setminus\{x_1, \ldots, x_n\}$, and so $g=1$ by hypothesis. Thus $U \subseteq \G(S)^0$, so $\int(\Is(\G(S))) \subseteq \G^0$ and thus $\int(\Is(\G(S))) = \G(S)^0$ because $\G(S)^0$ is open. Then $\G(S)$ is effective, which shows the `left-to-right' direction of the second statement.

Lastly assume that $\G(S)$ is effective, and suppose that $g \in G$ strongly fixes every letter in $X \setminus \{x_1, \ldots, x_n\}$, and consider the compact local bisection
$$U=(g, D(\varepsilon) \cap D(x_1)^c \cap \ldots \cap D(x_n)^c).$$
Note that if $\varepsilon\neq w\in D(\varepsilon) \cap D(x_1)^c \cap \ldots \cap D(x_n)^c$, then the first letter of $w$ is strongly fixed and so $g(w)=w$, whereas $g(\varepsilon)=\varepsilon$ is always true.  Therefore, we have that $U \subseteq \Is(\G(S))$. In particular,
$[g, \varepsilon] \in \int(\Is(\G(S)))=\G(S)^0$, forcing $g=1$ (as no element strictly below $g$ is defined at $\varepsilon$). This proves the claim.
\end{proof}

We prove, for completeness, a well-known proposition (going back to Nekrashevych~\cite{Nekcstar} in a different formulation for finite alphabets) describing the Hausdorff property for inverse semigroups associated to self-similar group actions.
To show that an inverse semigroup $S$ is Hausdorff, it suffices to show that, for each $s\in S$, the set of idempotents below $s$ is finitely generated as an order ideal; see~\cite[Proposition~2.2]{groupoidprimitive}.

\begin{Prop}
\label{p:describe.hausdorff}
Let $G$ be a self-similar group acting faithfully over an alphabet $X$ with $|X|\geq 2$.  Then the associated inverse semigroup $S$ is a Hausdorff inverse semigroup if and only if, for each $g\in G$,  there is a finite set $F_g\subseteq X^\ast$ such that $F_gX^\ast$ is the set of elements strongly fixed by $g$.
\end{Prop}
\begin{proof}
Let $s=ugv^\ast\in S$.  We claim that $ww^\ast \leq s$ (that is, $sww^\ast =ww^\ast$) if and only if  $u=v$ and $w=vz$ such that $g$ strongly fixes $z$.
First note that $s^\ast s= vv^\ast$ and so we must have that $v$ is a prefix of $w$ in order for $sww^\ast =ww^\ast$ to hold.  Moreover, if $w=vz$, then $sww^\ast = ug(z)g|_zz^\ast v^\ast=ug(z)g|_zw^\ast$.  The right hand side is $ww^\ast$ precisely when $g|_z=1$, $u=v$ and $g(z)=z$ (using $|g(z)|=|z|$).

So if $F_g\subseteq X^\ast$ denotes the unique minimal set of elements such that $F_gX^\ast$ is the set of words strongly fixed by $G$, then the set of idempotents below $ugv^\ast$ is empty if $u\neq v$, and otherwise consists of those $ww^\ast$ with $w\in vF_gX^\ast$.  In the latter case, this set is finitely generated as an order ideal if and only if $F_g$ is finite.
\end{proof}

Let $G$ be any self-similar group acting faithfully over an alphabet $A$ with $|A|\geq 2$. We can always obtain a faithful self-similar action of $G$ on an infinite alphabet satisfying Proposition~\ref{prop:effective}.
Put $X=A \times \mathbb N$, and define the self-similar action of $G$ over $X$ as the original action applied index-wise, that is, $g(a,i)= (g(a),i)$, $g|_{(a,i)}=g|_a$ for $a\in A$. If $Z$ is cofinite subset of $X$, then there is some index $j$ such that $A\times \{j\}\subseteq Z$  and so if $g$ strongly fixes $Z$, it must act trivially on $A^\ast$ (as seen by restricting the action of $g$ to $(A\times \{j\})^\ast$, which it strongly fixes). Hence $g$ is the identity by faithfulness. We call this action the \emph{countable inflation} of the self-similar group $G$.

Returning to the case of a general self-similar action of $G$ over an alphabet $X$, notice that $G$ acts freely on the right of the monoid $M=X^\ast G$ by multiplication and the set $X^*$ is a complete set of orbit representatives.  Hence, if $K$ is a commutative ring with unit, then $KM$ is a free right $KG$-module with basis $X^*$.  So each element of $KM$ can be uniquely written in the form $\sum_{w\in X^*}wa_w$ with $a_w\in KG$.  We shall exploit this fact in the sequel.    Observe that $a\in \K S$ is singular if and only if for all $u\in X^*$, there exists $v\in X^*$ with $auv=0$.

\begin{Prop}\label{zerofromG}
Let $G$ be a group with a self-similar action on $X^\ast$, let $M=X^\ast G$ be the associated left cancellative monoid and $S$ the inverse hull of $M$.
Let $a=\sum_{w\in X^*}wa_w\in KM$ with $a_w\in KG$.  Let $u\in X^*$.  Then $au=0$ if and only if $a_wu=0$ for all $w\in X^*$. Consequently, $a$ is singular in $\K S$ if and only if each $a_w$ is singular in $\K S$.
\end{Prop}
\begin{Proof}
As $au=\sum_{w\in X^*}wa_wu$, clearly if each $a_wu=0$ then $au=0$.  For the converse, let $n=|u|$ and note that, for $g\in G$, we have that $gu = g(u)g|_u$ and $|g(u)|=|u|=n$.   Thus $a_wu$ is of the form $\sum_{|v|=n}vb_{w,v}$ with $b_{w,v}\in KG$.  So if $au=0$, then  \[0=au = \sum_{w\in X^*}\sum_{|v|=n}wvb_{w,v}.\] In particular, notice that $wv=w'v'$ for $|v|=|v'|=n$ if and only if $w=w'$ and $v=v'$   and so we deduce that $b_{w,v}=0$ for all $w,v$ and thus $a_wu=0$ for all $w\in X^*$.

For the final statement, since the singular elements form an ideal, obviously if each $a_w$ is singular, then $a$ is singular.  The converse follows from what we have just proved because an element $c\in \K S$ is singular if and only if, for each word $v\in X^*$, there is a word $z\in X^*$ such that $cvz=0$.  But if $avz=0$, then $a_wvz=0$ for all $w\in X^*$.
\end{Proof}

\subsection{Simplicity can depend on the characteristic of the field}
Let $\mathcal P$ be a set of prime numbers (possibly empty).  In this section we construct a congruence-free inverse semigroup $S$ so that $\K S$ is simple if and only if the characteristic of $K$ does not belong to $\mathcal P$.

Let $C_p$ be a cyclic group of order $p$ and put $G=\bigoplus_{p\in \mathcal P} C_p\times C_p$; if $g\in G$, then $g_p\in C_p\times C_p$ will denote the $p$-component of $G$.  Notice that if $p\in \mathcal P$, then $G$ has $p+1$ subgroups of index $p$, namely those of the form $H\oplus\bigoplus_{q\in \mathcal P\setminus \{p\}} C_q\times C_q$ with $H\leq C_p\times C_p$ a subgroup of index (and order) $p$.  Let $A_p =\coprod_{[G:G_0]=p} G/G_0$ and let $A=\coprod_{p\in \mathcal P}A_p$.  The group $G$ acts on each $A_p$, and hence on $A$, in the natural way. Moreover, the action on $A$ is faithful since the projections of $G$ to the $C_p$ with $p\in \mathcal P$ separate points.    For $p\in \mathcal P$, let $\pi_{p}\colon G\to G$ be the composition \[G\twoheadrightarrow \bigoplus_{q\in \mathcal P\setminus \{p\}} C_q\times C_q\hookrightarrow G\] changing the $p$-component to $1$. We define a self-similar group action of $G$ on $A^*$ by using the above action on $A$ and putting $g|_a=\pi_p(g)$ for $a\in A_p$.  It is straightforward to verify that this provides a faithful self-similar group action by observing that the action of an element $g\in G$ on a word $w$ changes the leftmost letter $x$ of $w$ belonging to $A_p$ to $g_p(x)$ for any $p$ with $g_p\neq 1$,  and leaves the remaining letters as they were.  We consider the countable inflation of $G$ acting on
$X=A\times \mathbb N$ defined above.

Let $S$ be the inverse hull of the monoid $M=X^\ast G$ and let $K$ be a field.  Since $X$ is infinite, we have that $E(S)$ is strongly $0$-disjunctive and so the (contracted) universal groupoid of $S$ and its tight groupoid coincide.

We now study when an element $c=\sum_{g\in G}c_gg\in KG$ is singular in $\K S$.  Let \[\varphi(c) =\{p\in \mathcal P: \exists g\in \supp(c), g_p\neq 1\}\] be the set of primes `occurring' in $c$.  Since $g|_{(a,i)} = \pi_p(g)$ for $a\in A_p$, we see that $g|_{(a,i)}$ agrees with $g$, except perhaps in the component $p$, where it now becomes $1$.  Thus the following proposition is immediate.

\begin{Prop}\label{p:cutsupport}
Let $g\in G$ and suppose that $w\in X^*$ contains at least one letter $(a,i)$ with $a\in A_{p}$ for each prime $p\in \mathcal P$ with $g_p\neq 1$. Then $g|_w=1$.
\end{Prop}
\begin{Proof}
We induct on the number $n$ of primes $p$ with $g_p\neq 1$. There is nothing to prove if $n=0$.  Else, factor $w=u(a,i)v$ such that $(a,i)\in X$ is the first occurrence of a letter $(a,i)\in A_p\times \mathbb N$ with $g_p\neq 1$.  Then $g|_u=g$ and so $g|_w = (g|_u)|_{(a,i)v} = g|_{(a,i)v}=\pi_p(g)|_v=1$, where the last equality is by induction.
\end{Proof}

An immediate consequence of Proposition~\ref{p:cutsupport} is the following.

\begin{Cor}\label{c:full.action}
Let $c=\sum_{g\in G}c_g\in KG$ and $v\in X^*$.  Let $\varphi(c)\subseteq \{p_1,\ldots, p_n\}$ and $a_i\in A_{p_i}$ for $i=1,\ldots, n$.  Then \[c(a_1,j)(a_2,j)\cdots (a_n,j)v = \sum_{b_1\cdots b_n}\left(\sum_{g(a_1\cdots a_n)=b_1\cdots b_n}c_g(b_1,j)\cdots (b_n,j)v\right)\] with the $b_i\in A_{p_i}$.
\end{Cor}

Let $\rho\colon X\to A$ be the mapping $\rho(a,i)=a$ and note that $\rho$ induces a length-preserving, $G$-equivariant homomorphism $\rho\colon X^*\to A^*$.  Call a word $w\in X^*$ $c$-full if $\rho(w)$ contains a letter $a\in A_p$ for each $p\in \varphi(c)$.
Clearly any word is a prefix of a $c$-full word. By Proposition~\ref{p:cutsupport} if $w$ is $c$-full and $g,h\in \supp(c)$, then $g|_w=1=h|_w$, and it is not hard to see that $g(w)=h(w)$ if and only if $g,h$ agree on $\rho(w')$, where $w'$ is the word obtained from $w$ by erasing each letter except for the first occurrence of a letter from  each alphabet $A_p\times \mathbb N$ with $p\in \varphi(c)$.
  From this observation, the following criterion for singularity is straightforward, but we include a formal proof for completeness.

\begin{Prop}\label{p:singular.criterion.const}
Let $c=\sum_{g\in G}c_gg\in KG$ and suppose that $\varphi(c)\subseteq \{1,\ldots, p_n\}\subseteq \mathcal P$.  Then $c$ is singular if and only if, for all $\sigma\in S_n$ and $a_i,b_i\in A_{p_{\sigma(i)}}$, the equality \[\sum_{g(a_1\cdots a_n)=b_1\cdots b_n}c_g=0\] holds.
\end{Prop}
\begin{Proof}
Necessity is a straightforward consequence of Corollary~\ref{c:full.action}.  Let $a_1,\ldots, a_n\in A$ be as above. Then there exists $v\in X^*$ with $c(a_1,1)\cdots (a_n,1)v=0$.  Corollary~\ref{c:full.action} then yields that $\sum_{g(a_1\cdots a_n)=b_1\cdots b_n}c_g=0$.

For sufficiency, we induct on $n$.  If $n=0$, then $c=c_11$ and the assumed condition says that $c_1=0$ (where $a_1\cdots a_n$ is the empty word in this case).  Assume that it is true for $n-1$ element subsets of $\mathcal P$. We need to show that given $u\in X^*$, there exists $v\in X^*$ with $cuv=0$.  Notice that if $g_p=1$, then $g\cdot (a,i) = (a,i)g$ for any $a\notin A_p$ by construction.  Thus $c(a,i) = (a,i)c$ for any $a\notin \bigcup_{j=1}^n A_{p_j}$ and, moreover, $(a,i)cz=0$ if and only if $cz=0$ for any word $z\in X^*$ by Proposition~\ref{zerofromG}.  Thus it suffices to handle $u$ of the form $(a_1,i)u'$ where $a\in A_{p_j}$ for some $j=1,\ldots, n$. Reordering the indices, we may assume that $a_1\in A_{p_1}$.  Then \[c(a_1,i) = \sum_{b_1\in A_{p_1}}(b_1,i)\sum_{g(a_1)=b_1}c_gg|_{a_1}=\sum_{b_1\in A_{p_1}}(b_1,i)d_{b_1}\] where $d_{b_1} = \sum_{g(a_1)=b_1}c_g g|_{a_1}\in KG$.  Notice that  $\varphi(d_{b_1})\subseteq \{p_2,\ldots, p_n\}$ and so we can use the inductive hypothesis to show that each $d_{b_1}$ is singular.  Note that, for $\sigma$ a permutation of $2,\ldots, n$ and $a_i,b_i\in A_{p_{\sigma(i)}}$ we have that $g(a_1\cdots a_n) = g(a_1)g|_{a_1}(a_2\cdots a_n)$ and so the condition $g(a_1\cdots a_n) = b_1\cdots b_n$ is equivalent to $g(a_1)=b_1$ and $g|_{a_1}(a_2\cdots a_n) = b_2\cdots b_n$.  Thus the sum of the coefficients of $h\in \supp(d_{b_1})$ with $h(a_2\cdots a_n)=b_2\cdots b_n$ is precisely \[\sum_{g(a_1)=b_1}\left(\sum_{g|_{a_1}(a_2\cdots a_n)=b_2\cdots b_n}c_g\right) = \sum_{g(a_1\cdots a_n) = b_1\cdots b_n}c_g=0\] and so by induction each $d_{b_1}$ is singular and hence $c(a_1,i)$ is singular.  Thus we can find $v\in X^*$ with $0=c(a_1,i)u'v=cuv$ as required.
\end{Proof}

We shall only apply the proposition in the case $\varphi(c)=\{1,\ldots, p_n\}$; the more general formulation was just needed for the induction.
Our next lemma shows that restricting the support of an element of $\K S$ to $KG$ results in a singular element.  This is one of the few places that we use that there are infinitely many copies of the alphabet $A$ in $X$.

\begin{Lem}\label{l:restrict.to.group}
If $c'=\sum_{s\in S} c_ss\in \K S$ is singular, then so too is $c=\sum_{g\in G}c_gg$.
\end{Lem}
\begin{Proof}
We verify the condition in Proposition~\ref{p:singular.criterion.const} for $\varphi(c)=\{p_1,\ldots,p_n\}$.  Let $\sigma\in S_n$ and $a_i\in A_{p_{\sigma(i)}}$.  Choose an index $j$ such that no symbol $(a,j)$ appears in any $u,v\in X^*$ with $s=ugv^\ast\in \supp(c')$, $g\in G$.    Since $c'$ is singular, there exists $v\in X^*$ with $c'(a_1,j)\cdots (a_n,j)v=0$.  Recall that $M=X^*G$.  If $s\in \supp(c')\setminus M$, then $s(a_1,j)=0$ by choice of the index $j$.
Thus $0=c'(a_1,j)\cdots (a_n,j)v = \sum_{m\in M}c_mm(a_1,j)\cdots (a_n,j)v$.  Proposition~\ref{zerofromG} (with $w$ empty) and Corollary~\ref{c:full.action} then yield \[0=\sum_{g\in G}c_gg(a_1,j)\cdots (a_n,j)v = \sum_{b_1\cdots b_n}\sum_{g(a_1\cdots a_n)=b_1\cdots b_n}c_g(b_1,j)\cdots (b_n,j)v\] with the $b_i\in A_{p_{\sigma(i)}}$.  Thus $\sum_{g(a_1\cdots a_n)=b_1\cdots b_n}c_g=0$.  This completes the proof that $c$ is singular.
\end{Proof}

\begin{Cor}
\label{c:reduce.to.group}
The algebra $\K S$ contains a non-zero singular element if and only if $KG$ contains a non-zero element singular in $\K S$.
\end{Cor}
\begin{proof}
For the non-trivial direction, suppose that $a\in K_0S$ is a non-zero singular element.  Choose $s\in \supp(a)$ which is maximal with respect to the natural partial order.  Note that if $s=ugv^\ast$, then $s^\ast s = vv^\ast$ and so $v^\ast s^\ast sv=1$.  We claim that the coefficient of $1$ in $b=v^\ast s^\ast av$ is $a_s\neq 0$.  Since $b$ is singular and contains an element of $G$ in its support, Lemma~\ref{l:restrict.to.group} will then provide us the desired non-zero singular element in $KG$.

The coefficient of $1$ in $b$ is $\sum_{v^\ast s^\ast tv=1} a_t$.  In particular, $a_s$ appears in this sum.  If $t\in \supp(a)$ and $v^\ast s^\ast tv=1$, then $ss^\ast ts^\ast s = ss^\ast ss^\ast ts^\ast s=sv(v^\ast s^\ast tv)v^\ast = svv^\ast = ss^\ast s=s$ and so $s\leq t$.  Thus, by maximality of $s$, we obtain that $s=t$.  Hence the coefficient of $1$ in $b$ is indeed $a_s$.  This completes the proof.
\end{proof}

The following theorem proves the second item of Theorem~C.

\begin{Thm}
\label{thm.construction}
Let $K$ be a field and $S$ as constructed above.  Then $\K S$ is simple if and only if the characteristic of $K$ does not belong to $\mathcal P$.  In particular, if $\mathcal P$ consists of all primes, then $\K S$ is only simple in characteristic zero.
\end{Thm}
\begin{Proof}
Since $S$ is congruence-free, simplicity boils down to the singular ideal vanishing by Corollary~\ref{simplicity}.
Suppose first that the characteristic $p$ of $K$ belongs to $\mathcal P$.  Put $G_p=C_p\times C_p\leq G$.  Let $c=\sum_{g\in G_p}g\in KG$.  So $c_g=1$ if $g\in G_p$ and is $0$, otherwise.  Note that $\varphi(c)=\{p\}$.  We show that $c$ satisfies the criterion in Proposition~\ref{p:singular.criterion.const}.  Let $a\in A_p$.  Then the stabilizer $H$ of $a$ in $G_p$ is an index $p$ subgroup by construction, and hence $|H|=p$.  Thus if $b\in Ga=G_pa$, then there are exactly $p$ elements $g$ of $G_p$ with $g(a)=b$ and so $\sum_{g(a)=b} c_g=p=0$, as the characteristic of $K$ is $p$. If $b\notin Ga$, the sum is over the empty set and hence $0$.  Thus $c$ is singular by Proposition~\ref{p:singular.criterion.const}.

Now we assume that the characteristic of $K$ does not belong to $\mathcal P$ (it could be $0$).  If $\K S$ contains a non-zero singular element, then there is one belonging to $KG$ by Corollary~\ref{c:reduce.to.group}.  We show that if $c=\sum_{g\in G}c_gg$ is singular, then $c=0$ using Proposition~\ref{p:singular.criterion.const}.  Again, putting $G_p=C_p\times C_p$ for $p\in \mathcal P$, let $H = \bigoplus_{p\in \varphi(c)}G_p$ and note that $c\in KH$.  Moreover, $H$ is a finite abelian group whose order is not divisible by the characteristic of $K$ and hence $KH\cong K_1\times\cdots \times K_r$ where the $K_i$ are finite extensions of $K$ by Maschke's theorem.  Thus to show that $c=0$, it suffices to show that if $\chi\colon KH\to E$ is a homomorphism, with $E$ a finite extension of $K$, then $\chi(c)=0$.

First note that since finite subgroups of $E^\times$ are cyclic, we must have $\chi(G_p)$ is cyclic for each $p\in \varphi(c)$.  Hence, $\ker \chi|_{G_p}$ contains a subgroup $H_p\leq G_p$ of index $p$ for each $p\in \varphi(c)$.  There is an element $a_p\in A_p$ with stabilizer $H_p\oplus \bigoplus_{q\neq p} G_q$ in $G$ by construction of $A_p$.  Let $B=\bigoplus_{p\in \varphi(c)}H_p$.  Then $B\leq \ker\chi|_H$.  Moreover, if $\varphi(c) = \{p_1,\ldots, p_n\}$ and $a_i = a_{p_i}$, then $B$ is the stabilizer of the word $a_1\cdots a_n$ in $H$.  Indeed, $h(a_1\cdots a_n) = h_{p_1}(a_1)h_{p_2}(a_2)\cdots h_{p_n}(a_n)$ by definition of the action, since the $p_i$ are distinct.  Hence, if $g,h\in H$, then  $g(a_1\cdots a_n) = h(a_1\cdots a_n)$ if and only if $gB=hB$.  Let $T\subseteq H$ be a complete set of coset representatives for $H/B$.  Then, for each $t\in T$, we have \[\sum_{gB=tB}c_g=\sum_{g(a_1\cdots a_n) = t(a_1\cdots a_n)}c_g=0\] by Proposition~\ref{p:singular.criterion.const}.    Since $B\leq \ker \chi|_H$, whence $\chi$ is constant on cosets of $B$, we obtain
\[\chi(c) = \sum_{g\in H}c_g\chi(g) = \sum_{t\in T}\sum_{gB=tB}c_g\chi(g) = \sum_{t\in T}\chi(t)\sum_{gB=tB}c_g=0\] as was required.  We conclude that $c=0$ and so the singular ideal of $\K S$ is zero.  Since $S$ is congruence-free, the result follows from Corollary~\ref{simplicity}.
\end{Proof}

\begin{Rem}
 Note that $\K S=K\G_T(S)=K\G(S)$ and hence this result shows that the algebras of second countable minimal and effective groupoids can be non-simple over fields of any prescribed set of prime characteristics.  In light of Corollary~\ref{c:any.field.then.c}, the only thing that remains is to provide a second countable minimal effective groupoid whose algebra is not simple over any field, which we proceed to do shortly.
 \end{Rem}

Second countable effective ample groupoids are topologically principal and this played a critical role in proving sufficiency of the vanishing of the ideal of singular functions for simplicity in~\cite{nonhausdorffsimple}.  In~\cite[Example~6.4]{operatorsimple1} an example of an effective groupoid in which all isotropy groups are non-trivial was given, but it is not minimal and they asked whether minimal effective ample groupoids that are not topologically principal exist.  We provide here minimal and effective tight groupoids of inverse semigroups associated to self-similar group actions by uncountable groups in which each isotropy group is uncountable.  A Hausdorff example is given, as well as non-Hausdorff examples, the latter having simple algebras outside of a prescribed infinite set of prime characteristics.

As a warmup, we provide a Hausdorff example of a minimal and effective groupoid in which each isotropy group is uncountable. This groupoid will have a simple algebra over every field by the results of~\cite{operatorsimple1,operatorguys2,groupoidprimitive}.  It seems worth including such an example to show that the phenomenon of minimal and effective but not topologically principal ample groupoids can occur in the Hausdorff world.  Our example will be the tight groupoid of the inverse semigroup associated to an uncountable self-similar group action over a finite alphabet.

  Let $X$ be a finite set with at least $3$ elements and let $G=S_X^{\mathbb N}$, where $S_X$ is the symmetric group on $X$.  By the support of an element $\sigma\in G$, we mean the set of coordinates $i$ with $\sigma_i\neq 1$. The elements of finite support form a countable subgroup.  The action of $G$ on $X^\ast$ is given by $\sigma(x_1\cdots x_k) = \sigma_1(x_1)\sigma_2(x_2)\cdots \sigma_k(x_k)$ for $\sigma\in G$.  The action is self-similar with $\sigma|_x = T(\sigma)$ where $T$ is the shift map $T(\sigma)_i = \sigma_{i+1}$. Let $S$ be the corresponding inverse semigroup. Notice that if $\sigma\in G$ has infinite support, then every section of $\sigma$ is non-trivial and hence it strongly fixes no element.

\begin{Prop}
\label{nontopprinc}
The tight groupoid $\G_T(S)$ of  the inverse semigroup $S$ associated to $S_X^{\mathbb N}$ acting self-similarly on $X^\ast$ is Hausdorff, minimal and effective with every isotropy group uncountable for any finite set $X$ of cardinality at least $3$.
\end{Prop}
\begin{proof}
Note that since $X$ is finite, the tight groupoid $\G_T(S)$ of $S$ is the closed subgroupoid of $\G(S)$ with object set $X^\omega$. Since the inverse semigroup $S$ is congruence-free, the groupoid $\G_T(S)$ will be minimal by~\cite[Proposition~5.12]{groupoidprimitive}.  It is effective by~\cite[Section~17]{ExPadKatsura}.  We verify it is Hausdorff using Proposition~\ref{p:describe.hausdorff}. If $\sigma\in S_X^{\mathbb N}$ has infinite support, then it strongly fixes no element, whereas if $\sigma$ has finite support and $n$ is greatest with $\sigma_n\neq 1$, then the set of elements strongly fixed by $\sigma$ is $FX^\ast$ where $F$ is the finite set of words of length $n$ fixed by $\sigma$, as every section of $\sigma$ at a word of length $n$ is trivial and no section at a smaller length word is trivial.

Finally, we verify that each isotropy group is uncountable.
First note that $gww^\ast = g(w)g|_ww^\ast$ is an idempotent if and only if $g(w)=w$ and $g|_w=1$, that is, $g$ strongly fixes $w$.  In particular, elements of infinite support have no idempotents below them and hence have non-trivial germs at every infinite word. Let $n=|X|\geq 3$.  The stabilizer $H$ in $G$ of $w\in X^{\omega}$ is isomorphic to $S_{n-1}^{\mathbb N}$ (since the action of $G$ on $X^\omega$ is coordinate-wise) and hence is uncountable as $n\geq 3$. There is a homomorphism from $H$ to the isotropy group of $\G$ at $w$ given by $h\mapsto [h,w]$ and the kernel, by the above discussion, is contained in the set of elements of $H$ of finite support, which is countable (actually, the kernel is exactly the set of elements of $H$ of finite support).  Thus the image of $H$ in the isotropy group of $w$ is uncountable.   This completes the proof.
\end{proof}

For the non-Hausdorff examples, we modify our previous construction coming from a set of primes by replacing the direct sum by a direct product.
So let $\mathcal P$ be an infinite set of primes and let $\overline G= \prod_{p\in\mathcal P} C_p\times C_p$.  Notice that $G=\bigoplus_{p\in \mathcal P}C_p\times C_p$ is a subgroup of $\overline G$ and that $\overline G$ is uncountable because $\mathcal P$ is infinite.  Let $\pi_p\colon \overline G\to \overline G$ be the composition \[\overline G\twoheadrightarrow \prod_{q\in \mathcal P\setminus \{p\}} C_q\times C_q\hookrightarrow \overline G.\]  Let $A_p=\coprod_{[C_p\times C_p:H]=p}\overline G/(H\times \prod_{q\in \mathcal P\setminus \{p\}} C_q\times C_q)$ and let $A=\coprod_{p\in \mathcal P} A_p$.  This is essentially the same alphabet as before and $\overline G$ acts faithfully on $A$, extending the action of $G$ from before.  We again define a faithful self-similar action of $\overline G$ on $A$ by defining $g|_{a} = \pi_p(g)$ if $a\in A_p$.  The action of an element $g\in \overline G$ on a finite word $w\in A^*$ is again given by having $g_p$ act on the first occurrence of a letter from $A_p$ and leaving all other letters alone.

As before, we put $X=A\times \mathbb N$ and  consider the countable inflation of the self-similar group $\overline G$ acting on $X^\ast$.  The action of $G$ on $X^\ast$ is the same as before. Let $S$ be the inverse semigroup associated to this self-similar group.  Note that since $X$ is infinite, $\G(S)=\G_T(S)$.  We shall prove that $\K S$ is simple if and only if the characteristic of $K$ does not belong to $\mathcal P$.  Note that $\G(S)$ is minimal by~\cite[Proposition~5.2]{groupoidprimitive} and effective by our observation after Proposition~\ref{p:describe.hausdorff} that groupoids associated to countable inflations are always effective.

\begin{Prop}\label{p:isotropy.non-trivial}
Every isotropy group of $\G(S)$ is uncountable.
\end{Prop}
\begin{proof}
 Elements of $\overline G\setminus G$ do not strongly fix any element because $g|_w\neq 1$ for all $w\in X^*$ by construction of the action.  Thus no germ $[g,w]$ with $g\in \overline G\setminus G$ is trivial by the argument we saw in the proof of Proposition~\ref{nontopprinc}.  Next observe that the stabilizer in $\overline G$ of a finite word $w$ is of the form $H=H_{p_1}\times \cdots \times H_{p_k}\times \prod_{q\in \mathcal P\setminus \{p_1,\ldots, p_n\}}C_q\times C_q$ where $H_{p_i}$ has index $p_i$ in $C_{p_i}\times C_{p_i}$, and hence contains uncountably many elements of $\overline G\setminus G$.  Thus the corresponding isotropy group is uncountable by the same argument as in the proof of Proposition~\ref{nontopprinc} (the elements with trivial germs form a subgroup of the countable group $G$, but the stabilizer is uncountable).  Similarly, if $w$ is an infinite word, then there is a subset $\mathcal P'$ of $\mathcal P$ (perhaps infinite) such that the stabilizer of $w$ in $\overline G$ is of the form $\prod_{p\in \mathcal P'}H_p\times \prod_{q\in \mathcal P\setminus \mathcal P'}C_q\times C_q$ where $H_p\leq C_p\times C_p$ has index $p$.  Thus the stabilizer of $w$ is again uncountable and hence contains uncountably many elements of $\overline G\setminus G$, but the subgroup of elements stabilizing $w$ with trivial germs is countable, being contained in $G$.  Thus the isotropy group of $w$ in $\G(S)$ is uncountable.  This completes the proof.
\end{proof}

To prove our simplicity result, observe that Proposition~\ref{zerofromG}, of course, still applies, and Proposition~\ref{p:singular.criterion.const} still describes the elements of $KG$ which are singular in $\K S$ since the set of idempotents and the action of $G$ have not changed.  We have to make a minor modification of the proof of Lemma~\ref{l:restrict.to.group} in order to make it apply in our current setting.

\begin{Lem}\label{l:restrict.to.group.2}
If $c'=\sum_{s\in S} c_ss\in \K S$ is singular, then so too is $c=\sum_{g\in G}c_gg$.
\end{Lem}
\begin{Proof}
We verify the condition in Proposition~\ref{p:singular.criterion.const} for $\varphi(c)=\{p_1,\ldots,p_n\}$.  Let $\sigma\in S_n$ and $a_i\in A_{p_{\sigma(i)}}$.  Choose an index $j$ such that no symbol $(a,j)$ appears in any $u,v\in X^*$ with $s=ugv^\ast\in \supp(c')$, $g\in \overline G$.

Since $c'$ is singular, there exists $v\in X^*$ with $c'(a_1,j)\cdots (a_n,j)v=0$.  Recall that $M=X^*\overline G$.  If $s\in \supp(c')\setminus M$, then $s(a_1,j)=0$ by choice of the index $j$.
Thus $0=c'(a_1,j)\cdots (a_n,j)v = \sum_{m\in M}c_mm(a_1,j)\cdots (a_n,j)v$.  Proposition~\ref{zerofromG} (with $w$ empty) and Corollary~\ref{c:full.action} then yield
\begin{align*}
0&= \sum_{g\in \overline G}c_gg(a_1,j)\cdots (a_n,j)v\\ &=\sum_{g\in G}c_gg(a_1,j)\cdots (a_n,j)v+\sum_{g\in \overline G\setminus G}c_gg(a_1,j)\cdots (a_n,j)v \\ &= \sum_{b_1\cdots b_n}\sum_{g(a_1\cdots a_n)=b_1\cdots b_n}c_g(b_1,j)\cdots (b_n,j)v+\sum_{\overline G\setminus G}c_gg((a_1,j),\cdots, (a_n,j)v)g|_{(a_1,j)\cdots (a_n,j)v}
\end{align*}
 with the $b_i\in A_{p_{\sigma(i)}}$. But if $g\in \overline G\setminus G$, then $g|_w\neq 1$ for all $w\in X^*$ and so terms from the second summand cannot cancel out terms from the first summand.  Thus we must have $\sum_{g(a_1\cdots a_n)=b_1\cdots b_n}c_g=0$.  This completes the proof that $c$ is singular.
\end{Proof}

The proof of Corollary~\ref{c:reduce.to.group} now goes through \textit{mutatis mutandis}, using Lemma~\ref{l:restrict.to.group.2} in place of Lemma~\ref{l:restrict.to.group}, to show that $\K S$ has a non-zero singular element if and only if there is one in $KG$ and thus $\K S$ has a simple algebra precisely when the inverse semigroup in Theorem~\ref{thm.construction} does and so we have proved the following.

\begin{Thm}
\label{thm.construction2}
Let $K$ be a field and $\mathcal P$ an infinite set of primes.  Then there exists a minimal and effective groupoid $\G$ such that $K\G$ is simple if and only if the characteristic of $K$ does not belong to $\mathcal P$ and, moreover, every isotropy group of $\G$ is uncountable (and hence $\G$ is not topologically principal).  Moreover, $\G$ is the (contracted) universal groupoid of an inverse semigroup.
\end{Thm}

The simple algebras appearing in Theorem~\ref{thm.construction2} do not seem amenable to the techniques of~\cite{nonhausdorffsimple,Nekrashevychgpd}, which require the groupoids to be topologically principal.

\subsection{Minimal and effective ample groupoids with singular functions over all fields}
We now construct from self-similar group actions minimal and effective second countable ample groupoids which have non-zero singular functions over all fields.

In~\cite{groupoidprimitive}, the first author showed that a Hausdorff fundamental inverse semigroup with a $0$-disjunctive semilattice of idempotents has an effective tight groupoid.
In~\cite[Proposition~5.4]{LalondeMilan}, the authors claim that the interior of the isotropy subgroupoid of $\G_T(S)$ is $\G_T(Z_S(E(S)))$ so long as $E(S)$ is $0$-disjunctive.  This would, in particular, imply that  that any fundamental inverse semigroup with a $0$-disjunctive semilattice of idempotents has an effective tight groupoid, which we unfortunately found to be untrue for non-Hausdorff inverse semigroups.  (Their more general result about the interior of the isotropy subgroupoid is true for Hausdorff inverse semigroups, but needs to be argued along the lines of the fundamental case in~\cite{groupoidprimitive}.)   In examples where $S$ is fundamental and $E(S)$ is $0$-disjunctive, but the tight groupoid is not effective,  our uniqueness theorem (Theorem~\ref{uniqueness}) still applies despite the tight groupoid not being effective.

The following is an example of a congruence-free inverse semigroup with strongly $0$-disjunctive semilattice of idempotents and a non-effective universal (equivalently tight) groupoid.

Let $X=\{x,y\}$, and let $Z$ be an infinite set disjoint from $X$. Consider the self-similar action of the group $C_2=\{1,a\}$ (with identity $1$) over the alphabet $X \cup Z$ defined by $a(x)=y, a(y)=x, a(z)=z$ for $z\in Z$, and $a|_z=1$ for any $z \in (X \cup Y)^\ast$. Let $S$ be the inverse semigroup associated to this (faithful) self-similar action, and consider the universal groupoid $\G(S)$ of $S$.
Then $a \in G$ does not satisfy the assumption of Proposition~\ref{prop:effective} and, indeed, $[a, \varepsilon] \in (a, D(\empty)\setminus (D(x) \cup D(y)) \subseteq \int(\Is(\G(S)))$, but $[a, \varepsilon] \notin \G(S)^0$ as no element strictly below $a$ is defined at $\varepsilon$. This contradicts~\cite[Proposition~5.4]{LalondeMilan}.

By~\cite{groupoidprimitive}, the groupoid algebra, and therefore the contracted semigroup algebra of $S$ cannot be simple, as the groupoid is not effective.   And sure enough,
$$f=e-xx^\ast-yy^\ast-a+xy^\ast+yx^\ast$$
is in the singular ideal. If $z \in Z$, then
$fz=z-az=z-z=0$, while for instance
$$fx=ex-x-ax+y=x-x-y+y=0,$$
and $fy=0$ similarly. Hence clearly $f$ satisfies~\ref{singdef:right}.

In~\cite{nonhausdorffsimple}, the authors ask if there are any minimal and effective second countable ample groupoids whose $\mathbb C$-algebra is not simple, that is,  contains non-zero singular functions. We provide examples here.  After we had placed this paper on the ArXiv, Enrique Pardo informed us that Volodymr Nekrashevych (unpublished) had provided a self-similar group action over a finite alphabet such that the tight groupoid has non-zero singular functions over every field.  Nonetheless, our examples are universal groupoids of inverse semigroups and provide the first examples of congruence-free inverse semigroups with a strongly $0$-disjunctive semilattice, effective universal groupoid and a non-simple algebra over every field.

Our examples will arise from actions of weakly regular branch groups. Let us recall some definitions and notation, cf.~\cite{Branching}. Suppose $G$ is a self-similar group over the finite alphabet $A=\{a_1, \ldots, a_m\}$, and let $\Stab_1(G)=\{g \in G: g(a)=a \hbox{ for all } a \in A\}$. The map
\begin{gather*}\psi \colon \Stab_1(G) \to G^m\\
g \mapsto (g|_{a_1},\ldots, g|_{a_m})
\end{gather*}
is then an embedding.

A self-similar group is called \emph{spherically transitive} if acts transitively on the words of length $n$ for each $n\geq 0$.
A self-similar group $G$ is called a
\emph{weakly regular branch group} if it is spherically transitive and there is a non-trivial subgroup $H\leq \Stab_1(G)$ with $H^m \subseteq \psi(H)$. Note that $H$ is necessarily infinite. One sometimes says that $G$ is weakly branch over $H$.

Weakly regular branch groups include a lot of the well-known self-similar groups: the Grigorchuk group, the Gupta-Sidki groups, and the Basilica group~\cite{GNS,Branching}.  All three of these examples are amenable groups.  The following proposition establishes the third item of Theorem~C.

\begin{Prop}
If $G$ is a weakly regular branch group, then  universal groupoid of the inverse semigroup $S$ corresponding to the countable inflation of $G$ gives rise to a second countable minimal and effective ample groupoid $\G(S)$ with non-zero singular functions over every field.  Thus $K\G(S)\cong \K S$ is not simple for any field $K$.  The inverse semigroup $S$ is congruence-free and has a strongly $0$-disjunctive semilattice of idempotents.
\end{Prop}
\begin{proof}
Let the original alphabet be $A=\{a_1, \ldots, a_m\}$, and let $X=A \times \mathbb N$ be the infinite alphabet.
Let $S$ be the inverse semigroup corresponding to the self-similar action over the alphabet $X$.
Since $S$ is $0$-simple, we have that $\G(S)$ is minimal by~\cite[Proposition~5.2]{groupoidprimitive} and effective by the observation after Proposition~\ref{p:describe.hausdorff}. What remains is to show that it has non-trivial singular functions over any field $K$.

Choose a subgroup $H\leq G$  over which $G$ weakly branches for its action on $A^\ast$.  Note that $H$ stabilizes the letters of $A$, and hence $X$, by definition of weakly branching.
 Let $h_1, h_2 \in H \setminus \{1\}$ be two distinct elements.
Let $g_1, g_2, g_3\in H$ be elements with
\begin{align*}
\psi(g_1)&=(1,h_1,1, \ldots,1)\\
\psi(g_2)&=(h_2,1,1 \ldots,1)\\
\psi(g_3)&=(h_2,h_1,1, \ldots,1)
\end{align*}
 and note that $g_3=g_1g_2=g_2g_1$.
Then we claim
\[f=\chi_{(1,D(\varepsilon))}-\chi_{(g_1, D(\varepsilon))}-\chi_{(g_2, D(\varepsilon))}+\chi_{(g_3, D(\varepsilon))}\]
is a non-zero singular function in the Steinberg algebra $\stein S$.  It is not difficult to show that  $\supp(f) = \{[1,\varepsilon], [g_1,\varepsilon], [g_2, \varepsilon], [g_3,\varepsilon]\}$ and none of these is an isolated point, but we use Proposition~\ref{identify:singular} instead.  Note that $f$ is the element of $K\G(S)$ corresponding to $0\neq 1-g_1-g_2+g_3=(1-g_1)(1-g_2)=(1-g_2)(1-g_1)\in \K S$.   Also observe that $(1-g_1)(a_1,i) = (1-g_2)(a_2,i)=(1-g_1)(a_j,i)=0$ for $j\geq 3$ as $g_1$ strongly fixes $a_1$ and $a_j$ with $j\geq 3$ and $g_2$ strongly fixes $a_2$.  Thus $(1-g_1)(1-g_2)$ is in the singular ideal and hence $f$ is singular by Proposition~\ref{identify:singular} using that $\G(S)=\G_T(S)$.
This proves the claim.
\end{proof}

Note that in our proof we do not need the spherical transitivity property in the definition of a weakly branch group.

\section*{Acknowledgments}
The authors would like to thank Enrique Pardo for some helpful comments and a careful reading of the first version of this paper.  We also wish to thank the anonymous referee who suggested that simplicity of the essential algebra of a groupoid should be equivalent to minimality and topological freeness.

\def\malce{\mathbin{\hbox{$\bigcirc$\rlap{\kern-7.75pt\raise0,50pt\hbox{${\tt
  m}$}}}}}\def\cprime{$'$} \def\cprime{$'$} \def\cprime{$'$} \def\cprime{$'$}
  \def\cprime{$'$} \def\cprime{$'$} \def\cprime{$'$} \def\cprime{$'$}
  \def\cprime{$'$} \def\cprime{$'$}

\end{document}